\documentclass[10pt]{article}
\usepackage{amsmath}
\usepackage{amsfonts}
\usepackage{amssymb}
\usepackage{graphicx}
\usepackage{amsthm}
\usepackage{hyperref}
\usepackage{xcolor}
\usepackage{accents}
\usepackage[numbers,sort&compress,square]{natbib}

\usepackage[english]{babel}
\usepackage [autostyle, english = american]{csquotes}
\MakeOuterQuote{"}

\numberwithin{equation}{section}
\theoremstyle{plain}
\newtheorem*{theorem*}{Theorem}
\newtheorem*{lemma*}{Lemma}
\newtheorem{theorem}{Theorem}
\newtheorem{lemma}{Lemma}[section]

\newtheorem{proposition}[lemma]{Proposition}

\theoremstyle{definition}

\newtheorem{remark}[lemma]{Remark}

\newtheorem*{remark*}{Remark}
\newtheorem*{example*}{Example}
\newtheorem*{er*}{Examples and Remarks}

\usepackage{geometry}
\newcommand{\Hn}{{\mathcal H}}
\newcommand{\Wn}{{\mathcal W}}

\newcommand{\dthth}{\partial_{\theta \theta}} 
\newcommand{\dRth}{\partial_{R \theta}} 
\newcommand{\dth}{\partial_{\theta}} 
\newcommand{\dRR}{\partial_{RR}} 
\newcommand{\dR}{\partial_{R}}

\newcommand{\dx}{\partial_{x}} 
\newcommand{\dt}{\partial_{t}} 
\newcommand{\dy}{\partial_{y}} 
\newcommand{\dg}{\partial_{\gamma}}

\def\de{{\partial}}

\begin{document}

\title{ Strong Ill-Posedness in $L^\infty$ for the Riesz Transform Problem}
\author{Tarek M. Elgindi,   Karim R. Shikh Khalil}
\maketitle
\abstract{We prove strong ill-posedness in $L^{\infty}$ for linear perturbations of the 2d Euler equations of the form:
\[\partial_t \omega + u\cdot\nabla\omega = R(\omega),\] where $R$ is any non-trivial second order Riesz transform. 
Namely, we prove that there exist smooth solutions that are initially small in $L^{\infty}$ but become arbitrarily large in short time. Previous works in this direction relied on the strong ill-posedness of the linear problem, viewing the transport term perturbatively, which only led to mild growth. In this work we derive a nonlinear model taking all of the leading order effects into account to determine the precise pointwise growth of solutions for short time. Interestingly, the Euler transport term does counteract the linear growth so that the full nonlinear equation grows an order of magnitude less than the linear one. In particular, the (sharp) growth rate we establish is consistent with the global regularity of smooth solutions.}

\section{Introduction}

The Euler equations for incompressible flow are a fundamental model in fluid dynamics that describe the motion of ideal fluids:
\begin{equation}\label{Euler}
\begin{split}
\partial_t u &+  u\cdot\nabla u +\nabla p= 0 \\ 
&\nabla \cdot u=0. \\  
\end{split}
\end{equation}
 In this equation, $u$ is the velocity field and $p$ is the pressure of an ideal fluid flowing in $\mathbb{R}^2$. A key difficulty in understanding the dynamics of 2d Euler flows is the non-locality of the system due to the presence of the pressure term. 

 Defining the vorticity $\omega:= \nabla^\perp\cdot u$, it is insightful to study the Euler equations in vorticity form: 
\begin{equation}\label{EulerVorticty}
\begin{split}
\partial_t \omega &+  u\cdot\nabla \omega= 0, \\ 
&\nabla \cdot u=0 \\  
&u=\nabla^\perp \Delta^{-1} \omega.\\
\end{split}
\end{equation}
Because the $L^{\infty}$ norm of vorticity is conserved in the Euler equations in two dimensions, Yudovich \cite{Y} proved that there is a unique global-in-time solution to the Euler equation corresponding to every initial bounded and decaying vorticity. See also (\cite{Wo}, \cite{BKM}, \cite{HO},\cite{ Y},\cite{Ka}, \cite{MP},\cite{MB}). This bound on the $L^{\infty}$ norm is unfortunately unstable even to very mild perturbations of the equation ~\cite{CV,EM,EsharpLp}. To understand this phenomenon, we are interested in studying linear perturbations of the Euler equations in two dimensions as follows:
\begin{equation}\label{EulerRieszM}
\begin{split}
\partial_t u &+  u\cdot\nabla u +\nabla p= \begin{pmatrix}
0\\
u_1
\end{pmatrix}\\ 
&\nabla \cdot u=0 \\  
\end{split}
\end{equation}

\eqref{EulerRieszM} is a model for many problems in fluids dynamics that have a coupling with the Euler equations.   For instance, similar types of equations appear in viscoelastic fluids see~\cite{CK,EF,LM,CM} and in  magnetohydrodynamics see~\cite{BLW,H,CW,WZ}. Further, they also appear when studying stochastic Euler equation, see~\cite{GV}.

  Writing \eqref{EulerRieszM} in vorticity form, we get
\begin{equation}\label{EulerR}
\begin{split}
\partial_t \omega &+  u\cdot\nabla \omega= \partial_x u_1 \\ 
&\nabla \cdot u=0 \\  
&u=\nabla^\perp \Delta^{-1} \omega,\\
\end{split}
\end{equation} we observe that the challenge of studying these equations is that the right hand side of $\eqref{EulerR}$ can be written as the Riesz transform of vorticity $ \partial_x u_1= R(\omega)$,  which is unbounded on $L^{\infty}$. P. Constantin and V. Vicol  considered these equations with weak dissipation in \cite{CV}, and they proved  global well-posedness. However, without dissipation it is an open question whether these equations are globally well-posed.   In this work, we are interested in the question of  $L^{\infty}$ ill/well-posedness of the Euler equations with Riesz forcing and the local rate of $L^\infty$ growth. The first author and N. Masmoudi studied the Euler equations  with Riesz forcing in~\cite{EM},  where they proved that it is mildly ill-posed. This means  that there is a universal constant $c>0$ such that  for all $\epsilon>0$, there is $\omega_0\in C^\infty$ for which the unique local solution to $\eqref{EulerR}$ satisfies:

\begin{equation}\label{EMresult}
|\omega_0|_{L^\infty}\leq \epsilon, \,\, \text{but}  \,\, \sup_{t\in [0,\epsilon]} |\omega(t)|_{L^\infty}\geq c \end{equation}
The authors in~\cite{EM} conjectured   that the Euler  equations with Riesz forcing is actually strongly  ill-posed in $L^{\infty}$. Namely, that we can take $c$ in  $\eqref{EMresult}$ to be arbitrarily  large.  The goal of our work here is to show that indeed this is possible. To show this, we use the first author's  Biot-Savart law decomposition~\cite{E} to derive a leading order system for the Euler equations with Riesz  forcing. We then show that the leading order system is strongly ill-posed in $L^{\infty}$. Using this,  we can show that  the Euler equations with Riesz forcing is strongly   ill-posed by estimating the error  between the leading order system and the Euler with Riesz forcing system on a specific time interval.  

We should remark that the main application of the approach of the first author and N. Masmoudi in \cite{EM} was to prove ill-posedness of the Euler equation in the integer $C^k$ spaces, which was also proved independently by J. Bourgain and D. Li in \cite{BL}. Regarding the notion of mild ill-posedness in $L^{\infty}$ for models related to the Euler with Riesz forcing system, see  the work of J. Wu and J. Zhao in \cite{WZ} about the $2D$ resistive MHD equations.

\subsection{Statement of the main result}

\begin{theorem}
For any $\alpha,\delta>0$, there exists an initial data $\omega_0^{\alpha,\delta} \in C_c^{\infty}(\mathbb{R}^2)$  and   $T(\alpha)$ such that the corresponding unique global solution, $\omega^{\alpha,\delta}$, to \eqref{EulerR} is such that 
  at $t=0$ we have $$|\omega_0^{\alpha,\delta}|_{L^{\infty}}=\delta,$$ but for any  $0<t\leq T(\alpha)$ we have $$    \quad|\omega^{\alpha,\delta}(t)|_{L^{\infty}} \geq |\omega_0|_{L^\infty}+c \log (1+\frac{c}{\alpha}t),
  $$
where $T(\alpha)= c \alpha |\log(\alpha)|$ and $c>0$ is a constant independent of $\alpha.$
\end{theorem}

\begin{remark}
Note that at time $t=T(\alpha),$ we have that 
\[|\omega^{\alpha,\delta}|_{L^\infty}\geq c\log(c|\log\alpha|),\] which can be made arbitrarily large as $\alpha\rightarrow 0.$ Fixing $\delta>0$ small and then taking $\alpha$ sufficiently small thus gives strong ill-posedness for \eqref{EulerR} in $L^\infty.$ 
\end{remark}
\begin{remark}
As we will discuss below, we in fact establish upper and lower bounds on the solutions we construct so that on the same time-interval we have:
 $$    \quad|\omega^{\alpha,\delta}(t)|_{L^{\infty}} \approx |\omega_0|_{L^\infty}+c \log (1+\frac{c}{\alpha}t).
  $$
  This should be contrasted with the linear problem where the upper and lower bounds for the same data  come without the $\log:$
  \[|\omega^{\alpha,\delta}_{linear}(t)|_{L^\infty}\approx |\omega_0|_{L^\infty} + c(1+\frac{c}{\alpha}t).\]
\end{remark}

\begin{remark}
Our ill-posedness result applies to the equation:
\[\partial_t\omega+u\cdot\nabla\omega=R(\omega),\] where $R=R_{12}=\partial_{12}\Delta^{-1}.$ Note that a direct consequence of the result gives strong ill-posedness when $R=R_{11}$ or $R=R_{22}$ even though these are dissipative on $L^2.$ This can be seen just by noting that a linear change of coordinates can transform $R_{12}$ to a constant multiple of $R_{11}-R_{22}=R_{11}-Id$. The strong ill-posedness for the Euler equation with forcing by any second order Riesz transform (other than the identity) follows. We further remark that the same strategy can be used to study the case of general Riesz transforms though we do not undertake this here since the case of forcing by second order Riesz transforms is the most relevant for applications we are aware of (such as the 3d Euler equations, the Boussinesq system, visco-elastic models, MHD, etc.). 
\end{remark}

\subsection{Comparison with the linear equation and the effect of transport}

We now move to compare the result of this paper with the corresponding linear results and emphasize the regularizing effect of the non-linearity in this problem. The ill-posedness result of \cite{EM} relies on viewing \eqref{EulerR} as a perturbation of 
\begin{equation}\label{LinearEqn}\partial_t f = R(f).\end{equation} For this simple linear equation, it is easy to show that $L^\infty$ data can immediately develop a logarithmic singularity. Let us mention two ways to quantify this logarithmic singularity. One way is to study the growth of $L^p$ norms as $p\rightarrow\infty$. For the linear equation \eqref{LinearEqn}, it is easy to show that the upper bound:
\[|f(t)|_{L^p}\leq \exp(Ct) p |f_0|_{L^p}\] is sharp in the sense that we can find localized $L^\infty$ data for which the solution satisfies  \[|f(t)|_{L^p}\geq c(t)\cdot p.\] This can be viewed as approximating $L^\infty$ "from below." Similarly, the $C^\alpha$ bound for \eqref{LinearEqn},
\[|f(t)|_{C^\alpha}\leq \frac{\exp(Ct)}{\alpha} |f_0|_{C^\alpha}\] can also be shown to be sharp for short time in that we can find for each $\alpha>0$ smooth and localized data with $|f_0|_{C^\alpha}=1$ for which 
\[|f(t)|_{L^\infty}\geq \frac{c(t)}{\alpha}.\] The main result of \cite{EM} was that these upper and lower bounds remain unchanged in the presence of a transport term by a Lipschitz continuous velocity field. This is not directly applicable to our setting since the coupling between $\omega$ and $u$ is such that $u$ may not be Lipschitz even if $\omega$ is bounded. Interestingly, in \cite{EsharpLp}, it was shown that this growth could be significantly stronger in the presence of a merely bounded velocity field. 

All of the above discussion leads us to understand that the nature of the well/ill-posedness of \eqref{EulerR} will depend on the precise relationship between the velocity field and the linear forcing term in \eqref{EulerR}. In particular, for a natural class of data, we construct solutions to \eqref{EulerR} satisfying
\[|\omega|_{L^\infty} \approx 1+ \log(1+\frac{t}{\alpha}),\] for short time, which is the best growth rate possible in this setting. This should be contrasted with the corresponding growth rate for the linear problem \[|\omega_{lin}|_{L^\infty}\approx 1+ \frac{t}{\alpha}.\] In particular, the nonlinear term in \eqref{EulerR} actually tries to \emph{prevent} $L^\infty$ growth. Let us finally remark that the weak growth rate we found is consistent with the vorticity trying to develop a $\log\log$ singularity. It is curious that, in the Euler equation, vorticity with nearly $\log\log$ data are perfectly well-behaved and consistent with global regularity but with a triple exponential upper bound on gradients. Though establishing the global regularity rigorously remains a major open problem, this appears to be a sign that perhaps smooth solutions to \eqref{EulerRieszM} are globally regular.

\subsection{A short discussion of the proof}
The first step of the proof is to use the Biot-Savart law decomposition by the first author \cite{E} to derive a leading order model: 
\[\partial_t \Omega + \frac{1}{2\alpha}(L_s(\Omega)\sin(2\theta)+L_c(\Omega)\cos(2\theta))\partial_\theta\Omega= \frac{1}{2\alpha}L_s(\Omega),\] where the operators $L_s$ and $L_c$ are bounded linear operators on $L^2$ defined by $$
 L_s(f)(R)= \frac{1}{\pi}\int_{R}^{\infty} \int_0^{2\pi} \frac{f(s,\theta)}{s}  \sin(2\theta) \, d\theta \, ds  \quad \text{and} \quad L_c(f)(R)= \frac{1}{\pi}\int_{R}^{\infty} \int_0^{2\pi} \frac{f(s,\theta)}{s}  \cos(2\theta) \, d\theta \, ds.
 $$
 Essentially all we do here is replace the velocity field by its most singular part. Upon inspecting this model, we observe that the forcing term on the right hand side is purely radial while the direction of transport is angular. Upon choosing a suitable unknown, we thus reduce the problem to solving a transport equation for some unknown $f$:
\[\partial_t f + \frac{1}{2\alpha} L_s(f)\sin(2\theta)\partial_\theta f=0.\] Surprisingly, this reduced equation propagates the usual "odd-odd" symmetry even though the original system does not. The leading order model will then be strongly ill-posed if we can ensure that the solution of this transport equation satisfies that $\int_0^t L_s(f)$ can be arbitrarily large. One subtlety is that the growth of $L_s(f)$ enhances the transport effect, which in turn depletes the growth of $L_s(f)$. In fact, were the transport term to be stronger even by a log, the problem would \emph{not} be strongly ill-posed. By a careful study of the characteristics of this equation, we obtain a closed non-linear integro-differential equation governing the evolution of $L_s(f)$ (see equation \eqref{Lf}). We study this non-linear integro-differential equation and establish upper and lower bounds on $L_s(f)$  proving strong ill-posedness for the leading order equation; see section \ref{LOME} for more details.   Finally, we close the argument by estimating the error incurred by approximating the dynamics with the leading order model. An important idea here is to work on a time scale long enough to see the growth from the leading order model but short enough to suppress any potential stronger non-linear growth; see section \ref{ReminderEstS} for more details.

\subsection{Organization}
This paper is organized as follow: In section \ref{LOM}, we derive a leading order model for the Euler equations with Riesz forcing \eqref{EulerR} based on the first author's Biot-Savart law approximation~\cite{E}. Then, in section \ref{LOME}, we obtain a pointwise estimate on the leading order model which is the main ingredient in obtaining the strong ill-posedness result for the Euler with Riesz forcing system. In addition, in section  \ref{LOME}, we also obtain some estimates on the leading order model in suitable norms which will be then  used  in estimating the reminder term in section \ref{ReminderEstS}. After that, in section \ref{ElgindiEllipticEst} we will recall the first author's  Biot-Savart law decomposition obtained in~\cite{E}, and we will include a short sketch of the proof. In section \ref{EmbeddEst},  we will obtain some embedding estimates which will also be used in section \ref{ReminderEstS} for the reminder term estimates. Then, in section \ref{ReminderEstS}, we show  that the reminder term remains small which will then  allow us to prove the main result in section \ref{Main}. 

\subsection{Notation}

In this paper, we will be working in a form polar coordinates  introduced in~\cite{E}. Let $r$ be the  radial variable:
$$
r=\sqrt{x^2+y^2}
$$
and since we will be working with functions of the variable $r^{\alpha}$, where $0<\alpha<1$, we will use  $R$ to denote it:
$$
R=r^{\alpha}$$
We will use $\theta$ to denote the angle variable:
$$
\theta=\arctan{\frac{y}{x}}
$$

We will use $|f|_{L^{\infty}}$ and  $|f|_{L^2}$ to denote the usual $L^{\infty}$ and $L^{2}$ norms, respectively. In addition, we will  use $f_t$ or $f_{\tau}$ to denote the time variable. Further, in this paper, following~\cite{E}, we will be working on $(R, \theta) \in [0,\infty) \times  [0, \frac{\pi}{2}]$ where the $L^2$ norm will be with measure $dR \, d\theta$ and not  $R\, dR \, d\theta$. 

We define the weighted $\Hn^{k}( [0,\infty) \times  [0, \frac{\pi}{2}])$ norm as follows:

$$
|f|_{\dot{\Hn}^m}=  \sum_{i=0}^m  |\dR^i \dth^{m-i} f|_{L^2}+ \sum_{i=1}^m |R^i \dR^i \dth^{m-i} f|_{L^2}
$$

$$
|f|_{\Hn^k}= \sum_{m=0}^{k} |f|_{\dot{\Hn}^m}
$$

We also define $\Wn^{k,\infty}$ norm as follows:
$$
|f|_{\dot{\Wn}^{m,\infty}}=  \sum_{i=0}^m  |\dR^i \dth^{m-i} f|_{L^{\infty}}+ \sum_{i=1}^m |R^i \dR^i \dth^{m-i} f|_{L^{\infty}}
$$

$$
|f|_{\Wn^{k,\infty}}= \sum_{m=0}^{k} |f|_{\dot{\Wn}^{m,\infty}}
$$

Throughout this paper, we will use the following notation to define the following operators:
$$
L(f)(R)=\int_R^{\infty}\frac{f(s)}{s} \, ds
$$
and by adding a subscript $L_s$ or $L_c$, we denote the project onto $\sin(2 \theta)$ and $\cos(2 \theta)$ respectively. Namely, 
$$
 L_s(f)(R)= \frac{1}{\pi}\int_{R}^{\infty} \int_0^{2\pi} \frac{f(s,\theta)}{s}  \sin(2\theta) \, d\theta \, ds  \quad \text{and} \quad L_c(f)(R)= \frac{1}{\pi} \int_{R}^{\infty} \int_0^{2\pi} \frac{f(s,\theta)}{s}  \cos(2\theta) \, d\theta \, ds
 $$

\section{Leading Order Model}\label{LOM}

In this section, we will derive a leading order model for the Euler equation with Riesz forcing:
\begin{equation}\label{EulerRiesz}
\begin{split}
\partial_t \omega &+  u\cdot\nabla \omega= \partial_x u_1 \\ 
&\nabla \cdot u=0 \\  
&u=\nabla^\perp \Delta^{-1} \omega\\
\end{split}
\end{equation}

To do this, we follow~\cite{E} and we write the equation in a form of polar coordinates. Namely, we set  $r=\sqrt{x^2+y^2}$, $R=r^\alpha$, and $\theta= \arctan{\frac{y}{x}}$. We will the rewrite the equation \eqref{EulerRiesz} in the new functions $ \omega(x,y)=\Omega(R,\theta)$ and $ \psi(x,y)= r^2 \Psi(R,\theta)$ with  $u=\nabla^\perp \psi $, where 
$
u_1=- \dy \psi
$, and  
$u_2= \dx \psi
$.  

\textbf{Equations of $u$ in terms of $\Psi$}
$$
u_1=-r(2 \sin(\theta) \Psi+ \alpha \sin(\theta )R  \, \dR  \Psi +\cos(\theta) \dth \Psi)
$$
$$
u_2=r(2 \cos(\theta) \Psi+ \alpha \cos(\theta )R  \, \dR  \Psi -\sin(\theta) \dth \Psi)
$$

\textbf{Evolution Equation for $ \Omega$}
\begin{equation*}
\begin{split}
\partial_t{\Omega} +  \Big( -\alpha R \dth \Psi  \Big)\dR \Omega + \Big( 2 \Psi  + \alpha R \dR \Psi \Big)  \dth \Omega &=  \big(-2 \alpha R \sin(\theta) \cos(\theta) -\alpha^2 R \sin(\theta) \cos(\theta) \big) \dR \Psi \\  &+\big( -1+2 \sin^2(\theta) \big) \dth \Psi +\big(- \alpha R \cos^2(\theta) + \alpha R \sin^2(\theta)\big) \dRth\Psi \\ &-  \big(\alpha^2R^2 \sin(\theta) \cos(\theta) \big)\dRR \Psi + \big( \sin(\theta) \cos(\theta)  \big)\dthth\Psi \\
\end{split}
\end{equation*}

\textbf{The elliptic equation for $\Delta( r^2 \Psi(R,\theta)) =\Omega(R,\theta)$}
$$
4 \Psi + \alpha^2 R^2  \dRR \Psi + \dthth \Psi +(4 \alpha +\alpha^2) R \dR \Psi=\Omega(R,\theta)
$$

Now using  the first author's  Biot-Savart decomposition  \cite{E}, see section \ref{ElgindiEllipticEst} for more details, by defining the operators  
$$
 L_s(\Omega)(R)=\frac{1}{\pi} \int_{R}^{\infty} \int_0^{2\pi} \frac{\Omega(s,\theta)}{s}  \sin(2\theta) \, d\theta \, ds  \quad \text{and} \quad L_c(\Omega)(R)=\frac{1}{\pi} \int_{R}^{\infty} \int_0^{2\pi} \frac{\Omega(s,\theta)}{s}  \cos(2\theta) \, d\theta \, ds
 $$

we have 
$$
\Psi(R,\theta)=-\frac{1}{4 \alpha}  L_s(\Omega) \sin(2 \theta)- \frac{1}{4 \alpha}  L_c(\Omega) \cos(2 \theta)+ \text{lower order terms}
$$

Thus, if we ignore the $\alpha$ terms in the evolution equation, we obtain 
\begin{equation} \label{FundEq}
\begin{split}
\partial_t{\Omega} +  \Big( 2 \Psi   \Big)  \dth \Omega &=  \Big( -1+2 \sin^2(\theta) \Big) \dth \Psi +\Big( \sin(\theta) \cos(\theta)  \Big)\dthth\Psi \\
\end{split}
\end{equation}

Now  we consider $\Psi$ of the form 

$$
\Psi= -\frac{1}{4 \alpha}  L_s(\Omega) \sin(2 \theta)- \frac{1}{4 \alpha}  L_c(\Omega) \cos(2 \theta)
$$

and plug it into the evolution equation, we have  
\begin{equation*}
\begin{split}
\partial_t{\Omega} -  \Big(  \frac{1}{2 \alpha}   L_s(\Omega)\sin(2\theta)+\frac{1}{2 \alpha} L_c(\Omega) \cos(2\theta) 
 \Big)  \dth \Omega =   &-\big( \cos(2\theta)\big) \big(- \frac{1}{2 \alpha} L_s(\Omega) \cos(2\theta) +\frac{1}{2 \alpha} L_c(\Omega) \sin(2\theta) 
 \big)  \\
  &+ \big( \frac{1}{2}\sin(2 \theta)  \big) \big( \frac{1}{ \alpha}  L_s(\Omega) \sin(2\theta)+\frac{1}{ \alpha}  L_c(\Omega) \cos(2\theta)   \big)  \\
\end{split}
\end{equation*}

which simplifies to 

\[\partial_t{\Omega} - \Big(  \frac{1}{2 \alpha}   L_s(\Omega)\sin(2\theta)+\frac{1}{2 \alpha} L_c(\Omega) \cos(2\theta) 
 \Big)  \dth \Omega =\frac{1}{2 \alpha} L_s(\Omega)
\]

In order to work with positive solutions and have the angular trajectories moving to the right, we make the change $\Omega\rightarrow-\Omega$ and get the final model:

\begin{equation}\label{OmegaModal}
\begin{split}
\partial_t{\Omega} +  \Big(  \frac{1}{2 \alpha}   L_s(\Omega)\sin(2\theta)+\frac{1}{2 \alpha} L_c(\Omega) \cos(2\theta) 
 \Big)  \dth \Omega &=\frac{1}{2 \alpha} L_s(\Omega).\\
\end{split}
\end{equation}

\noindent We now move to study the dynamics of solutions to \eqref{OmegaModal}.

\begin{proposition} \label{OmegaModalPStat}Let $\Omega$ be a solution to the leading order model 
\begin{equation}\label{OmegaModalP}
\begin{split}
\partial_t{\Omega} +  \Big(  \frac{1}{2 \alpha}   L_s(\Omega)\sin(2\theta)+\frac{1}{2 \alpha} L_c(\Omega) \cos(2\theta) 
 \Big)  \dth \Omega &=\frac{1}{2 \alpha} L_s(\Omega)\\
\end{split}
\end{equation}
with initial data of the form  $\Omega|_{t=0}=f_0 (R) \sin(2 \theta)$ then we can write $\Omega$ as follow: 
\begin{equation}
\Omega= f + \frac{1}{2 \alpha} \int_0^t  L_s(f_\tau) d \tau
\end{equation}
where $f$ satisfies the following transport equation:

\begin{equation} 
\begin{split}
\partial_t{f} +  \frac{1}{2 \alpha}  \sin(2\theta)  L_s(f)  \dth f &=0\\
\end{split}
\end{equation}
 
\end{proposition}

\begin{proof}
The righthand side term of \eqref{OmegaModalP} is  radial, and hence if we take the inner product with $\sin(2\theta)$ it will be zero. 
Now if write $\Omega$ as:  $$ \Omega_t(R,\theta)= f_t (R,\theta)+ \frac{1}{2 \alpha} \int_0^t  L_s(\Omega_\tau)(R) d \tau$$
and consider it to be a solution to  $\eqref{OmegaModalP}$, we obtain that $f$ satisfies the following:
\begin{equation} \label{fModel}
\begin{split}
\partial_t{f_t} +  \Big(  \frac{1}{2 \alpha}   L_s(f_t)\sin(2\theta)+\frac{1}{2 \alpha} L_c(f_t) \cos(2\theta) 
 \Big)  \dth f_t &=0\\
\end{split}
\end{equation}
Here we used that $L_s(\Omega_\tau)(R)$ is a radial function. Notice that $\eqref{fModel}$ is a transport equation that  preserves odd symmetry. Now if we set: 
 $$f_t^s=\int_0^{2\pi} f_t(R,\theta) \sin(2\theta) d\theta \quad \text{and} \quad \Omega_t^s=\int_0^{2\pi} \Omega_t(R,\theta) \sin(2\theta) \, d \theta,  $$
 we notice that $f_t^s$ and $\Omega_t^s$ will satisfy the same equation. Thus, if we start with the same initial conditions $f_0=\Omega_0$, then $$f_t^s=\Omega_t^s \quad \text{for all  } t $$ Thus, we  have $L_s(\Omega_t)=L_s(f_t)$, and hence 
$$
\Omega_t= f_t + \frac{1}{2 \alpha} \int_0^t  L_s(f_\tau) d \tau
$$

Now since the initial data which  we are considering  have odd symmetry, it suffices to consider the following transport  equation:
 \begin{equation} \label{fModel2}
\begin{split}
\partial_t{f_t} +  \frac{1}{2 \alpha}  \sin(2\theta)  L_s(f_t)  \dth f_t &=0\\
\end{split}
\end{equation}

\end{proof}
\section{Leading Order Model Estimate}\label{LOME}
The purpose of this section is to obtain $L^{\infty}$ estimates for the leading order model which is the main ingredient in obtaining  the ill-posedness result for the the Euler with Riesz forcing system. This will be done in subsection \ref{PLOME} in  three steps: Lemma \ref{LeadinIneq}, Lemma \ref{ApprOp}, and Proposition \ref{LeadingEst}. Then in subsection \ref{ELOMnorms}, we will obtain some estimate for the leading order model which will be useful in reminder estimates in section  \ref{ReminderEstS} . 
\subsection{Pointwise Leading Order Model Estimate}\label{PLOME}

\begin{lemma} \label{LeadinIneq} Let $f$ be a solution to the following transport equation: 
\begin{equation} \label{fM}
\begin{split}
\partial_t{f} +  \frac{1}{2 \alpha}  \sin(2\theta)  L_s(f)  \dth f &=0\\
\end{split}
\end{equation}
with initial data $f|_{t=0}=f_{0}(R) \sin(2 \theta)$, then we have the following estimate on the operator   $L_s(f)$: 
\begin{equation} \label{ULbound}
c_1  \int_{R}^{\infty} \frac{f_0(s)}{s} \exp(-\frac{1}{\alpha} \int_{0}^t L_s(f_\tau)(s) \, d\tau) \, d s
 \leq L_{s}(f_t)(R) \leq c_2 \int_{R}^{\infty} \frac{f_0(s)}{s} \exp(-\frac{1}{\alpha} \int_{0}^t L_s(f_\tau)(s) \, d\tau) \, d s
\end{equation}
where $c_1$ and $c_2$ are independent of $\alpha$
\end{lemma}

\begin{proof}

To prove this, we consider the following variable change. For  $ \theta \in [0, \frac{\pi}{2})$, let $\gamma$ be defined as follows 
$$
\gamma:= \tan(\theta) \implies \frac{d \gamma}{d \theta}= \sec^2(\theta),  \,  \,\, \text{and}  \, \, \sin(2\theta)= \frac{2\gamma}{1+\gamma^2}
$$
Applying chain rule, we  rewrite $\eqref{fM}$ in the $(R,\gamma)$ variables:

\begin{equation} \label{ModelGamma}
\dt f_t +  \frac{1}{\alpha}\gamma \,  L_{s}(f_t)(R) \, \dg f=0
\end{equation}

with initial date 
$$
f|_{t=0}=f_{0}(R) \sin(2\theta)=f_{0}(R) \frac{2\gamma}{1+\gamma^2} $$

Let $\phi_t(\gamma)$ be the flow map associated with  $\eqref{ModelGamma}$, so we have 

$$
\frac{d \phi_t(\gamma)}{ dt}= \frac{1}{\alpha}  \phi_t(\gamma) L_s(f_t) \implies  \phi_t(\gamma)= \gamma \exp(\frac{1}{\alpha}  \int_{0}^t L_s(f_\tau) \, d\tau)
$$

Thus,
$$
  \phi^{-1}_t(\gamma)= \gamma \exp(-\frac{1}{\alpha} \int_{0}^t L_s(f_\tau) \, d\tau)
$$

 Hence, we now  write the solution to $\eqref{ModelGamma}$ as follows: 
 
 $$
 f_t(R,\gamma)=f_0(R, \phi^{-1}_t(\gamma))=f_{0}(R) \frac{2 \phi^{-1}_t(\gamma)}{1+ \phi^{-1}_t(\gamma)^2}=f_{0}(R) \, \frac{2  \, \gamma \exp(-\frac{1}{\alpha} \int_{0}^t L_s(f_\tau) \, d\tau)}{1+ \gamma^2 \exp(-\frac{2}{\alpha} \int_{0}^t L_s(f_\tau) \, d\tau)}
 $$

 Now we consider the operator $L_s$ in the $(R,\gamma) \in [0,\infty)\times [0,\frac{\pi}{2})$ variables:

$$
L_{s}(f_t)(R)=\frac{1}{\pi} \int_{R}^{\infty} \frac{1}{s} \int_{0}^{\infty} f_t(s,\gamma) \,   \frac{ 2 \gamma}{(1+ \gamma^2)^2} \,  d \gamma  \, d s
$$

Plugging the expression for $f_t$, we have

\begin{equation}\label{Lf}
L_{s}(f_t)(R)=\frac{1}{\pi} \int_{R}^{\infty} \frac{1}{s} \int_{0}^{\infty}  \ f_{0}(s)\, \frac{  \,  \exp(-\frac{1}{\alpha}  \int_{0}^t L_s(f_\tau)(s) \, d\tau)}{1+ \gamma^2 \exp(-\frac{2}{\alpha}  \int_{0}^t L_s(f_\tau)(s) \, d\tau)} \,    \frac{ 4 \gamma^2}{(1+ \gamma^2)^2} \,  d \gamma \, ds
\end{equation}

Now since $0 \leq \exp(-\frac{2}{\alpha}  \int_{0}^t L_s(f_\tau)(s) \, d\tau) \leq 1$, we have a upper and lower bound on the operator on  $L_{s}(f_t)(R)$ with constants $c_1,c_2$ independent of $\alpha$ (In fact, these constants can be explicitly computed). Namely, \begin{equation*} 
c_1  \int_{R}^{\infty} \frac{ f_{0}(s)}{s} \exp(-\frac{1}{\alpha} \int_{0}^t L_s(f_\tau)(s) \, d\tau) \, d s
 \leq L_{s}(f_t)(R) \leq c_2 \int_{R}^{\infty} \frac{ f_{0}(s)}{s} \exp(-\frac{1}{\alpha} \int_{0}^t L_s(f_\tau)(s) \, d\tau) \, d s
\end{equation*}

Thus, we have our desired inequalities.

\end{proof}

\begin{lemma} \label{ApprOp}
Define the operator 
\begin{equation} \label{ApxL}  
\hat{L}(f_t)(R):= \int_{R}^{\infty} \frac{f_{0}(s)}{s } \exp(-\frac{1}{\alpha}\int_{0}^t \hat{L}(f_s)(s) \, d\tau) \, ds
\end{equation}
Then we have 
$$
\int_{0}^t \hat{L}(f_\tau)(R) \, d \tau=  2 \alpha \log(1+  \frac{t}{2 \alpha } \,  L(f_0)(R))
$$
where $L(f_0)(R)=\int_R^{\infty} \frac{f_0(s)}{s} \,ds $
 
\end{lemma}

\begin{proof}

We introduce  $g_t(R):= \exp(-\frac{1}{\alpha}\int_{0}^t \hat{L}(f_\tau)(R) \, d\tau) \,$ and $K(R):=\frac{f_0(R)}{R}$, then   the operator $\hat{L}$  can be rewritten as:

\begin{equation}\label{Lkg} 
\hat{L}(f_t)(R)=\int_{R}^{\infty} K(s) g_t(s) \, ds
\end{equation} 

Now taking time derivative of $\eqref{Lkg}$, and using that $\dt g_t(R)= -2 g_t(R) \, \int_{R}^{\infty} K(s) g_t(s) \, ds$, we can obtain: 
$$
\dt  \hat{L}(f_t)= -\frac{1}{2 \alpha} (\hat{L}(f_t))^2
$$
which can be solved explicitly: 
\begin{equation} \label{Lexplicit}
\hat{L}(f_t)(R)=  \frac{L(f_0)(R)}{1+\frac{t}{2 \alpha}\, L(f_0)(R)}
\end{equation}
and then it follows that 
$$
\int_{0}^t \hat{L}(f_t)(R)d \tau= 2 \alpha  \log(1+\frac{t}{2 \alpha}\, L(f_0)(R))
$$

\end{proof}

\begin{proposition}\label{LeadingEst} Let $f$ be a solution to the following transport equation: 
\begin{equation} 
\begin{split}
\partial_t{f} +  \frac{1}{2 \alpha}  \sin(2\theta)  L_s(f)  \dth f &=0\\
\end{split}
\end{equation}
with initial data $f|_{t=0}=f_{0}(R) \sin(2 \theta)$, then we have the following estimate on the operator   $L_s(f)$: 
\begin{equation} \label{ULboundT}
\frac{ 2\alpha}{c_1} \log(1+\frac{c_1}{ 2\alpha}t\, L(f_0)(R))) \geq \int_{0}^t L_{s}(f_\tau)(R) \geq \frac{ 2\alpha}{c_2}  \log(1+\frac{c_2}{ 2\alpha}t\, L(f_0)(R))
\end{equation}
where $c_1$ and $c_2$ are independent of $\alpha$
\end{proposition}

\begin{proof} 

In the section, we will use the bounds in $\eqref{ULbound}$, Namely

\begin{equation}\label{ULbound2}
c_1   \int_{R}^{\infty} \frac{f_0(s)}{s} \exp(-\frac{1}{\alpha} \int_{0}^t L_s(f_\tau)(s) \, d\tau) \, d s
 \leq L_{s}(f_t)(R) \leq c_2 \int_{R}^{\infty} \frac{f_0(s)}{s}\exp(-\frac{1}{\alpha} \int_{0}^t L_s(f_\tau)(s) \, d\tau) \, d s, 
\end{equation}
to obtain and upper and lower estimate on $\int_0^t L_{s}(f)$.  As before we set:

 $$g_t(R)=\exp(-\frac{1}{\alpha}\int_{0}^t L_{s}(f_\tau)(R) \, d\tau)\quad \text{and} \quad K(R)=\frac{f_0(R)}{R}$$
Using $\eqref{ULbound2}$ we can obtain that

\begin{equation}\label{ULbound5}
-\frac{c_1}{2 \alpha}     \bigg( \int_{R}^{\infty}g_t(s) K(s)  \, ds  \bigg)^2
 \geq  \dt  \int_{R}^{\infty} g_t(s) K(s)  \, ds   \geq  -  \frac{c_2}{2 \alpha}     \bigg( \int_{R}^{\infty}g_t(s) K(s)  \, du  \bigg)^2
\end{equation}

Similar to Lemma \ref{ApprOp}, we define 
$$
L_s(f_t)(R):= \int_{R}^{\infty} g_t(s) K(s) \, ds
$$

Now from $\eqref{ULbound5}$,  we have 
$$
- \frac{c_1}{2 \alpha}    (L_s(f_t)(R))^2
 \geq  \dt L_s(f_t)(R)  \geq  -  \frac{c_2}{2 \alpha}   (L_s(f_t)(R))^2
$$

Thus,

\begin{equation}\label{ULWbound}
 \frac{L(f_0)(R)}{1+ \frac{c_1}{2 \alpha}  \, t\, L(f_0)(R)}
 \geq  L_s(f_t)(R)   \geq  \frac{L(f_0)(R)}{1+ \frac{c_2}{2 \alpha}  \, t\, L(f_0)(R)}
\end{equation}

which will give us that:

$$
\frac{ 2\alpha}{c_1} \log(1+\frac{c_1}{ 2\alpha}t\, L(f_0)(R))) \geq \int_{0}^t L_{s}(f_\tau)(R) \geq \frac{ 2\alpha}{c_2}  \log(1+\frac{c_2}{ 2\alpha}t\, L(f_0)(R))
$$

and this completes the proof 

\end{proof} 

\subsection{  Estimate for the leading order model in $\Wn^{k,\infty}$ and $\Hn^k$ norms } \label{ELOMnorms}
The purpose of this subsection is to obtain some estimate on the leading order model in $\Wn^{k,\infty}$ and $\Hn^k$ norms. These will be used to estimate the size of the reminder term in section \ref{ReminderEstS}. First we will obtain  estimates on $\Psi_2$ in Lemma \ref{LPsi2Est}. Then in Lemma \ref{LOmega2Est}, we will obtain  estimates on $\Omega_2$.

\begin{lemma}\label{LPsi2Est} Let $\Omega_2$ to solution to the  leading order model:
\begin{equation*}
\begin{split}
\partial_t{\Omega_2} +  \Big(  \frac{1}{2 \alpha}   L_s(\Omega_2)\sin(2\theta)+\frac{1}{2 \alpha} L_c(\Omega_2) \cos(2\theta) 
 \Big)  \dth \Omega_2 &=\frac{1}{2 \alpha} L_s(\Omega_2)\\
\end{split}
\end{equation*}
with initial data $\Omega_2|_{t=0}=f_{0}(R) \,  \sin(2\theta)$
, where $f_{0}(R)$ is smooth with compactly support. Consider 
$$
\Psi_2= \frac{1}{4 \alpha}  L_s(\Omega_2) \sin(2 \theta)+ \frac{1}{4 \alpha}  L_c(\Omega_2) \cos(2 \theta)
$$
Then, we have the following estimates on $\Psi_2$: 
\begin{equation} \label{Psi2Est}
|\Psi_2|_{\Wn^{k+1,\infty}}\leq  \frac{c_{k}}{ \alpha}, \quad |\Psi_2|_{\Hn^{k+1}}\leq  \frac{c_{k}}{ \alpha}
\end{equation}

where $c_{k}$  depends on the initial conditions and is independent  of $\alpha$

\end{lemma}

\begin{proof}

Recall that from Proposition  \ref{OmegaModalPStat}, we can write $\Omega_2$ as follows:

$$
  \Omega_2= f + \frac{1}{2 \alpha} \int_{0}^t L_s(f_{\tau}) \, d\tau ,
$$

 and  since the initial data is odd in $\theta$, we have
   $$\Psi_2=\frac{1}{4 \alpha}L_s(\Omega_t) \sin(2 \theta)=\frac{1}{4 \alpha}L_s(f_t) \sin(2 \theta)$$
  
  To estimate the size of $\Psi_2$, from \eqref{Lf}, we have  
    $$
L_{s}(f_t)(R)=\int_{R}^{\infty} \frac{1}{s} \int_{0}^{\infty}  \ f_{0}(s)\, \frac{  \,  \exp(-\frac{1}{\alpha}  \int_{0}^t L_s(f_\tau)(s) \, d\tau)}{1+ \gamma^2 \exp(-\frac{2}{\alpha}  \int_{0}^t L_s(f_\tau)(s) \, d\tau)} \,    \frac{ 4 \gamma^2}{(1+ \gamma^2)^2} \,  d \gamma \, ds
$$

 Using \eqref{ULbound}, we have  
 
 $$
| \Psi_2|_{L^{\infty}} \leq \frac{c}{\alpha} \int_{R}^{\infty} \frac{\ f_{0}(s)}{s} \, ds  \leq \frac{c_0}{\alpha} 
 $$
 
For $\dth \Psi_2 $, it is clear that we have 

$$
|\dth \Psi_2 |_{L^{\infty}} \leq  \frac{c_0}{\alpha} 
$$

 where, similarly,  $c_0$ depends on the initial condition.

Now for $\dR \Psi_2$, we have 

$$
\dR \Psi_2=\frac{1}{4 \alpha} \dR L_s(f_t) \sin(2 \theta)
$$

Thus,

$$
\dR L_{s}(f_t)(R)= - \frac{1}{R} \int_{0}^{\infty}  \ f_{0}(R)\, \frac{  \,  \exp(-\frac{1}{\alpha}  \int_{0}^t L_s(f_\tau)(R) \, d\tau)}{1+ \gamma^2 \exp(-\frac{2}{\alpha}  \int_{0}^t L_s(f_\tau)(R) \, d\tau)} \,    \frac{ 4 \gamma^2}{(1+ \gamma^2)^2} \,  d \gamma \,
$$
and similarly, we have

 $$| \, \dR \Psi_2 |_{L^{\infty}} \leq \frac{c}{ \alpha}$$

Now the estimate on  $R  \, \partial_{R } \Psi_2$ follows from the estimate on $\dR \Psi_2$ and the fact that the initial data have compact support. Thus, 

 $$|  R \,  \dR \Psi_2 |_{L^{\infty}} \leq \frac{c}{ \alpha}$$

For higher order derivative, we can obtain the estimate following the same steps. Hence, we have
\begin{equation*} 
|\Psi|_{\Wn^{k+1,\infty}}\leq  \frac{c_{k}}{ \alpha}
\end{equation*}
The $\Hn^{k}$ estimates also follows using the same steps.  

 \begin{equation*}
|\Psi|_{\Hn^{k+1}}\leq  \frac{c_{k}}{ \alpha}
\end{equation*}
 
\end{proof}

In the following Lemma, we will obtain the $\Hn^{k}$ estimates on $\Omega_2$. Here we will use Lemma \ref{LPsi2Est} and transport estimates. 
 \begin{lemma}\label{LOmega2Est} Let $\Omega_2$ to solution to the  leading order model:
\begin{equation*}
\begin{split}
\partial_t{\Omega_2} +  \Big(  \frac{1}{2 \alpha}   L_s(\Omega_2)\sin(2\theta)+\frac{1}{2 \alpha} L_c(\Omega_2) \cos(2\theta) 
 \Big)  \dth \Omega_2 &=\frac{1}{2 \alpha} L_s(\Omega_2)\\
\end{split}
\end{equation*}
with initial data $\Omega_2|_{t=0}=f_{0}(R) \,  \sin(2\theta)$, where $f_{0}(R)$ is smooth with compactly support. Then, we have the following estimates on  $\Omega_2$: 

\begin{equation} \label{Omega2Est}
|\Omega_2|_{\Hn^{k}}  \leq c_k  \, e^{ \frac{c_k}{ \alpha}t }
\end{equation}

where $c_{k}$  depends on the initial conditions and is independent  of $\alpha$

\end{lemma}

\begin{proof}

Recall that from Proposition  \ref{OmegaModalPStat} we can write $\Omega_2$ as follows:

$$
  \Omega_2= f + \frac{1}{2 \alpha} \int_{0}^t L_s(f_{\tau}) \, d\tau ,
$$
  
where $f$ satisfies the following transport equation: 
$$
\partial_t f_t + 2 \Psi_2 \dth f_t=0
$$
When we consider the derivatives of $\Omega_2$, the transport term $f$ will dominates the radial term $\frac{1}{2 \alpha} \int_{0}^t L_s(f) \, d\tau $. Thus, it suffices to consider the $\Hn^{k}$ estimates on $f$ which will follow from the standard $L^2$ estimate for the transport equation. Thus, Since we have   
$$
 \partial_t f_t + 2 \Psi_2 \dth f_t=0 \implies  \partial_t  \dth f_t + 2   \dth\Psi_2 \dth f_t + 2 \Psi_2 \dthth f_t =0
$$

Hence, 

$$
  |\dth f_t|_{L^2} \leq     |\dth f_0|_{L^2}  \,  e^{\int_0^t |\dth\Psi_2|_{L^{\infty}}} 
$$
From $\eqref{Psi2Est}$  we have  $|\dth \Psi_2 |_{L^{\infty}} \leq  \frac{c_0}{ \alpha}$. Thus,  applying Gronwall inequality,  we have 

\begin{equation}\label{dthf}
|\dth f_t|_{L^2} \leq    |\dth f_0|_{L^2}  e^{ \frac{c_0}{ \alpha} t}
\end{equation}

To obtain $\Hn^k$ estimates, we need to estimate terms of the form $R^k \dR^k$. We will show how to obtain the $R \dR$ estimate, and for general  $k$, it will follow similarly.  Thus, similar to $L^2$ estimate for $\dth f$ case, since  
$$
 \partial_t f_t + 2 \Psi_2 \dth f_t=0
$$

we have 

$$
\partial_t  \dR f_t + 2   \dR\Psi_2 \dth f_t + 2 \Psi_2 \de_{R \theta} f_t =0
$$
and thus,

$$
\partial_t  |R \dR f_t|_{L^2} \leq 2 |R \dR \Psi_2 |_{L^{\infty}} |\dth f|_{L^2} +  |\dth\Psi_2|_{L^{\infty}} |R\dR f_t|_{L^2}
$$
Now from $\eqref{Psi2Est}$, $\eqref{dthf}$, and applying Gronwall inequality  we have

\begin{equation*}
\begin{split}
 |R \dR f_t|_{L^2}  & \leq \big(  |R\dR f_0|_{L^2} +  |\dth f_0|_{L^2}  e^{ \frac{c_0}{ \alpha}t}   \big)  e^{ \frac{c_0}{\alpha} t}
\end{split}
\end{equation*}
Hence, 

$$
 |f(t)|_{\Hn^{1}} \leq  |f_0|_{\Hn^{1}} \,   e^{ \frac{c_1}{\alpha} t}
$$

which implies that 

$$| \Omega_2(t)|_{\Hn^1}   \leq  \  | \Omega_2(0)|_{\Hn^1} e^{ \frac{c_1}{ \alpha}t } $$

Similarly, using $ \eqref{Psi2Est}$, the transport estimate, and following the same steps as above, we can obtain for the general $\Hn^k$  estimates. Hence 
 $$| \Omega_2|_{\Hn^k}   \leq  \  | \Omega_2(0)|_{\Hn^k} e^{ \frac{c_k}{ \alpha}t } $$

\end{proof}

\section{Elliptic Estimate}\label{ElgindiEllipticEst}
The purpose of this section is to recall the first author's  Biot-Savart law decomposition~\cite{E} which is used here to derive the leading order model. In this section, we highlight the main ideas in the proof, and for more details,  see~\cite{E} and~\cite{DE}. We remark that this is also related to the Key Lemma of A. Kiselev and V. \v{S}ver\'{a}k, see also the work of the first author \cite{Eremarks}, and the first author and I. Jeong \cite{EJ} for generalization.

\begin{proposition}\label{EllipticEst}(\cite{E})
Given $\Omega \in H^k$ such that for every $R$ we have  $$\int_{0}^{2\pi} \Omega(R,\theta) \sin(n \theta)  \, d\theta=\int_{0}^{2\pi} \Omega(R,\theta) \cos(n \theta) \, d\theta= 0$$
for $n=0,1,2$ then the unique solution to $$4 \Psi +  \dthth \Psi + \alpha^2 R^2  \dRR \Psi  +(4 \alpha +\alpha^2) R \dR \Psi=\Omega(R,\theta)$$ satisfies

\begin{equation} \label{EllipticEstMain}
|\dthth \Psi|_{H^k} \, + \, \alpha | R \dRth \Psi|_{H^k}  \, + \,   \alpha^2|R^2 \dRR \Psi|_{H^k} \leq C_{k} |\Omega|_{H^k}
\end{equation}

where $C_k$ is \textbf{independent} of $\alpha$. In addition, we have the weights estimate

\begin{equation} \label{EllipticEstW}
|\dthth D^k_R (\Psi)|_{L^2} \, + \, \alpha | R \dRth  D^k_R (\Psi )|_{L^2}  \, + \,   \alpha^2|R^2 \dRR D^k_R (\Psi)|_{L^2} \leq C_{k} |D^k_R(\Omega)|_{L^2}
\end{equation}

where $C_k$ is \textbf{independent} of $\alpha$. Recall that $D_R=R\dR$

\end{proposition}
\begin{proof}
 
 First, we will show how to obtain $\eqref{EllipticEstMain}$.  Since $\Omega$ is orthogonal to $ \sin(n \theta)$ and   $\cos(n \theta) $ for $n=0,1,2$. Then, $\Psi$ must also be orthogonal to $ \sin(n \theta)$ and   $\cos(n \theta) $ for $n=0,1,2$. Consider the elliptic  equation, and we consider $L^2$ estimate. 
$$4 \Psi +  \dthth \Psi + \alpha^2 R^2  \dRR \Psi  +(4 \alpha +\alpha^2) R \dR \Psi=\Omega(R,\theta)$$
Taking the inner product with $\dthth \Psi$ and integrating by parts, we have obtain 
$$-4 |\dth\Psi|^2_{L^2} +  |\dthth \Psi|^2_{L^2} - \alpha^2   |\dth \Psi|^2_{L^2} + \alpha^2|R\dRth \Psi|^2_{L^2} + \frac{(4 \alpha +\alpha^2)}{2}|\dth \Psi|^{2}_{L^2} \leq |\Omega|_{L^2} |\dthth \Psi|_{L^2}$$

Now  by assumption, we have  

$$
\Psi(R,\theta) =\sum_{n \geq 3} \Psi_n(R)e^{in \theta}  
$$

and hence 

$$
 |\dth\Psi|^2_{L^2} \leq  \frac{1}{9}  |\dthth \Psi|^2_{L^2} 
$$

Using the above inequality, we can show that  

$$\frac{5}{9} |\dthth \Psi|^2_{L^2}  + \alpha^2|R\dRth \Psi|^2_{L^2} + \frac{(4 \alpha -\alpha^2)}{2}|\dth \Psi|^{2}_{L^2} \leq |\Omega|_{L^2} |\dthth \Psi|_{L^2}$$

and thus we have that
 $$
 |\dthth \Psi|_{L^2}  \leq   C_0 |\Omega|_{L^2} 
$$
where $C_0$ is independent of $\alpha$. The estimate for the $R^2 \dRR \Psi$ term will follow similarly. We can also obtain the $H^k$ estimates by following the same strategy. To obtain $\eqref{EllipticEstW}$ estimates, recall that $D_R=R \dR$ and we notice that we can write the elliptic equation in the following form: 

$$4 \Psi +  \dthth \Psi + \alpha^2 D^2_R (\Psi)  +4 \alpha  \, D_R( \Psi) =\Omega(R,\theta)$$
From this, we observe that the $D_R$ operator commutes with the elliptic equation, and hence $\eqref{EllipticEstW}$ estimates will follow from \eqref{EllipticEstMain}. 

\end{proof}

\begin{theorem}(\cite{E})
Given $\Omega \in H^k$ where $\Omega$ has the form of  $$\Omega(R,\theta)= f(R) \sin(2 \theta)   \, \, \,  \,\,   \Big( \Omega(R,\theta)= f(R) \cos(2 \theta)  \Big)$$
 then the unique solution to $$4 \Psi +  \dthth \Psi + \alpha^2 R^2  \dRR \Psi  +(4 \alpha +\alpha^2) R \dR \Psi=\Omega(R,\theta)$$ is  
$$\Psi=-\frac{1}{4 \alpha}  L(f)(R) \sin(2 \theta)+ \mathcal{R}(f) \, \, \,  \,\, \,   \Big( \Psi=-\frac{1}{4 \alpha}  L(f)(R)  \cos(2 \theta) + \mathcal{R}(f)  \Big)$$
where 
$$L(f)(R)=\int_{R}^{\infty} \frac{f(s)}{s} \, ds$$
and
$$|\mathcal{R}(f)|_{H^k} \leq c |f|_{H^k} $$

where $c$ is independent of $\alpha$

\end{theorem}
\begin{proof}
Consider the case where $\Omega(R,\theta)= f(R) \sin(2 \theta)$, the case where $\Omega(R,\theta)= f(R) \cos(2 \theta)$ can be handled similarly. In this case $\Psi(R,\theta)$ will be of the form  $\Psi(R,\theta)=\Psi_2(R) \sin(2 \theta)$. Where  $\Psi_2(R)$  will satisfy the following ODE:
$$
\alpha^2 R^2 \dRR \Psi_2+(4 \alpha+ \alpha^2) R \dR \Psi_2=f(R)
$$
we can solve the ODE and obtain 

$$
\dR \Psi_2(R)=\frac{1}{\alpha^2} \frac{1}{R^{\frac{4}{\alpha}+1}} \int_0^R \frac{f(s)}{s^{1-\frac{4}{\alpha}}} \, ds
$$
Now using that $\Psi_2(R) \rightarrow 0$ as $R \rightarrow \infty$, we obtain 

$$
 \Psi_2(R)=-\frac{1}{\alpha^2} \int_{R}^{\infty} \frac{1}{\rho^{\frac{4}{\alpha}+1}} \int_0^{\rho} \frac{f(s)}{s^{1-\frac{4}{\alpha}}} \, ds
$$
 By integrating by parts, it follows that
 
 $$
 \Psi_2(R)=-\frac{1}{4 \alpha} \int_R^{\infty}\frac{f(s)}{s} \, ds- \frac{1}{4 \alpha} \frac{1}{R^{\frac{4}{\alpha}}} \int_0^R \frac{f(s)}{s^{1-\frac{4}{\alpha}}} \, ds :=-\frac{1}{4 \alpha}L(f)(R)+ \mathcal{R}(f)
$$
Using Hardy-type inequality one can show that: 
$$|\mathcal{R}(f)|_{L^2} \leq c |f|_{L^2} $$

where $c$ is independent of $\alpha$

\end{proof}

\section{ Embedding estimate in terms of $\Hn^k$ norm}\label{EmbeddEst}
In this section we consider some  embedding estimate in the $\Hn^k$ norm which will be used in section \ref{ReminderEstS}. These estimates will be used various times as we estimate the reminder term.  Recall that the $\Hn^k$ norm is defined as follows: 
$$
|f|_{\dot{\Hn}^m}=  \sum_{i=0}^m  |\dR^i \dth^{m-i} f|_{L^2}+ \sum_{i=1}^m |R^i \dR^i \dth^{m-i} f|_{L^2}
$$
$$
|f|_{\Hn^k}= \sum_{m=0}^{k} |f|_{\dot{\Hn}^m}
$$

\begin{lemma} \label{OmegaEmbed} Let $\Omega \in \Hn^N$, where $N \in \mathbb{N}$, then we have 

\begin{equation} \label{ROmegaLinfty}
|R^k \dR^k \dth^m \Omega|_{L^{\infty}}  \leq c_{k,m}  |\Omega|_{\Hn^{k+m+2}} \quad \text{for any} \quad  k+m+2 \leq N
\end{equation}

\end{lemma}

\begin{proof}

Using Sobolev  embedding, we have  
 \begin{equation*}
\begin{split}
|R^k \dR^k \dth^m \Omega|_{L^{\infty}} &\leq c_{k,m}  |R^k \dR^k \dth^m \Omega|_{H^{2}_{R,\theta}}  \\         
 \end{split}
\end{equation*}
where $H^{2}_{R,\theta} $ is the standard $H^2$ norm in $R$ and $\theta$. 
When considering  the second derivative terms of $R^k \dR^k \dth^m \Omega$, for the angular derivatives term, we have $|R^k \dR^k \dth^{m+2} \Omega|_{L^2}  \leq  |\Omega|_{\Hn^{k+m+2}}$. Now for the radial derivatives, we have three cases. Considering  the case when the two radial derivatives land on $\dR^k \dth^m \Omega$, we have   $$|R^{k} \dR^{k+2} \dth^{m} \Omega|_{L^2}  \leq  |R^{k+2} \dR^{k+2} \dth^{m} \Omega|_{L^2}+|  \dR^{k+2} \dth^{m} \Omega| \leq |\Omega|_{\Hn^{k+m+2}}$$
where the last inequality follows from the definition of the $\Hn^{N}$ norm. The other two cases follow in a similar way.

\end{proof}

We will also need some embedding estimates for the stream function $\Psi$ in terms of $\Omega$.

\begin{lemma}\label{PsiEmbed}  Let $\Omega \in \Hn^N$, where $N \in \mathbb{N}$,  satisfying the same conditions as in Proposition \ref{EllipticEst}, then for the solution $\Psi$ of $$4 \Psi +  \dthth \Psi + \alpha^2 R^2  \dRR \Psi  +(4 \alpha +\alpha^2) R \dR \Psi=\Omega(R,\theta),$$ we have 
 \begin{equation} \label{PsiInfty}
 |\dR^k \dth^m \Psi|_{L^{\infty}}  \leq  c_{k,m}   |\Omega|_{\Hn^{k+m+1}} 
\end{equation}
for $k,m \in \mathbb{N}$  with $k+m+1 \leq N$  
\end{lemma}

\begin{proof}

As in Lemma \ref{OmegaEmbed}, applying the Sobolev  embedding, we have 
 \begin{equation*}
\begin{split}
 | \dR^k \dth^m \Psi|_{L^{\infty}} &\leq    \,  c_{k,m} | \dR^k \dth^m \Psi|_{H^{2}_{R,\theta}} \\         
 \end{split}
\end{equation*}
From the elliptic estimates in  Proposition \ref{EllipticEst}, for any $i,n \in \mathbb{N}$, we have   
 \begin{equation} \label{PsiL2Omega}
\begin{split}
 | \dR^i \dth^{n} \Psi|_{L^2}  \leq  c_{i,n} \, |\Omega|_{\Hn^{i+n-1}}
 \end{split}
\end{equation}

Thus, to bound $ | \dR^k \dth^m \Psi|_{H^{2}_{R,\theta}}$, we take $\Omega$ to be in $\Hn^{k+m+1}$. Hence, we have 

 \begin{equation} \label{PsiInfty}
 |\dR^k \dth^m \Psi|_{L^{\infty}}  \leq  c_{k,m}   |\Omega|_{\Hn^{k+m+1}} 
\end{equation}

\end{proof}

\begin{lemma}\label{RPsiEmbed}  Let $\Omega \in \Hn^N$, where $N \in \mathbb{N}$,  satisfying the same conditions as in Proposition \ref{EllipticEst}, then for the solution $\Psi$ of $$4 \Psi +  \dthth \Psi + \alpha^2 R^2  \dRR \Psi  +(4 \alpha +\alpha^2) R \dR \Psi=\Omega(R,\theta),$$ we have 
 \begin{equation} \label{RPsiInfty}
 |R^k \dR^k \dth^m \Psi|_{L^{\infty}}  \leq  c_{k,m}   |\Omega|_{\Hn^{k+m+1}}
\end{equation}
for $k,m \in \mathbb{N}$  with $k+m+1 \leq N$  
\end{lemma}

\begin{proof}

As in Lemma \ref{OmegaEmbed}, applying the Sobolev  embedding, we have 
 \begin{equation*}
\begin{split}
 |R^k \dR^k \dth^m \Psi|_{L^{\infty}} &\leq   \,  c_{k,m} |R^k \dR^k \dth^m \Psi|_{H^{2}_{R,\theta}} \\         
 \end{split}
\end{equation*}
From the elliptic estimates in  Proposition \ref{EllipticEst}, for any $i,n \in \mathbb{N}$, we have   
 \begin{equation} \label{PsiL2}
\begin{split}
 | \dR^i \dth^{n} \Psi|_{L^2} \leq  c_{i,n} \, |\dR^{i} \dth^{n-1} \Omega|_{L^2} \leq  c_{i,n} \, |\Omega|_{\Hn^{i+n-1}}
 \end{split}
\end{equation}
and
 \begin{equation} \label{RPsiL2}
\begin{split}
 |R^i \dR^i \dth^{n} \Psi|_{L^2}  \leq c_{i,n}  \, |\Omega|_{\Hn^{i+n-1}}
 \end{split}
\end{equation}

Thus, if we look at the second derivative terms of $R^k \dR^k \dth^m \Psi$, we can use the above inequalities to obtain the desired estimate. For the angular derivative term, we have $ |R^k \dR^k \dth^{m+2} \Psi|_{L^2}  \leq c_{k,m} \,  |\Omega|_{\Hn^{k+m+1}} $. When considering the radial derivative terms, we have three terms. For $R^{k} \dR^{k+2} \dth^{m} \Psi$ term, applying \eqref{PsiL2} and \eqref{RPsiL2}, we have 
$$  |R^{k} \dR^{k+2} \dth^{m} \Psi|_{L^2}  \leq   |R^{k+2} \dR^{k+2} \dth^{m} \Psi|_{L^2}+ |  \dR^{k+2} \dth^{m} \Psi| \leq c_{k,m}\, |\Omega|_{\Hn^{k+m+1}}
$$
The other terms can be handled in similar way.  Hence, we have our desired result.

\end{proof}

\section{ Reminder estimate } \label{ReminderEstS}

In this section, we obtain an error estimate on the remaining terms in the Euler with Riesz forcing. Recall that $\Omega$ satisfies the following evolution equation: 
\begin{equation*}
\begin{split}
\partial_t{\Omega} +  \Big( -\alpha R \dth \Psi  \Big)\dR \Omega + \Big( 2 \Psi  + \alpha R \dR \Psi \Big)  \dth \Omega &=  \big(2 \alpha R \sin(\theta) \cos(\theta) +\alpha^2 R \sin(\theta) \cos(\theta) \big) \dR \Psi \\  &+\big( 1-2 \sin^2(\theta) \big) \dth \Psi +\big( \alpha R \cos^2(\theta) + \alpha R \sin^2(\theta)\big) \dRth\Psi \\ &+  \big(\alpha^2R^2 \sin(\theta) \cos(\theta) \big)\dRR \Psi - \big( \sin(\theta) \cos(\theta)  \big)\dthth\Psi \\
\end{split}
\end{equation*}

and and the elliptic equation is the following:
$$
4 \Psi + \alpha^2 R^2  \dRR \Psi + \dthth \Psi +(4 \alpha +\alpha^2) R \dR \Psi=\Omega(R,\theta)
$$

From section \ref{LOM}, the leading order model for the Euler with Riesz forcing equation satisfies the following:

\begin{equation}
\begin{split}
\partial_t{\Omega_2} +  \big( 2  \Psi_2   \big)  \dth \Omega_2 &=   \big( -1+2 \sin^2(\theta) \big) \dth \Psi_2 +  \big( \sin(\theta) \cos(\theta)  \big)\dthth\Psi_2 \\
\end{split}
\end{equation}
Where 
$$
\Psi_2(R,\theta)=
\frac{1}{4 \alpha}  L_s(\Omega) \sin(2 \theta)+ \frac{1}{4 \alpha}  L_c(\Omega) \cos(2 \theta)
$$

Now set  $ \Omega_r =\Omega- \Omega_2$ to be the reminder term for the vorticity, and similarly  set $ \Psi_r =\Psi- \Psi_2$ to be the reminder term for the stream function. Thus, we have that reminder, $\Omega_r$, satisfies  the following evolution equation:   
\begin{equation}\label{ReminderEq}
\begin{split}
\partial_t{\Omega_r} &+  \Big( -\alpha R (\dth \Psi_2  +\dth \Psi_r ) \Big)(\dR \Omega_2 + \dR \Omega_r)+ \Big( 2  \Psi_2   \dth \Omega_r + 2  \Psi_r   \dth \Omega_2 + 2   \Psi_r   \dth \Omega_r  \Big) 
 \\& +    \Big(\alpha R  (\dR \Psi_2+ \dR\Psi_r) \Big)(  \dth \Omega_2 + \dth \Omega_r)= \big(2 \alpha R \sin(\theta) \cos(\theta) +\alpha^2 R \sin(\theta) \cos(\theta) \big) \big(\dR \Psi_2 + \dR \Psi_r \big) 
  \\  &+ \big( 1-2 \sin^2(\theta) \big) \dth \Psi_r + \alpha\big( R \cos^2(\theta) -  R \sin^2(\theta)\big)( \dRth\Psi_2+\dRth\Psi_r)
   \\ &+   \alpha^2\big(R^2 \sin(\theta) \cos(\theta) \big)(\dRR \Psi_2+ \dRR \Psi_r) -  \big( \sin(\theta) \cos(\theta)  \big)\dthth\Psi_r \\
\end{split}
\end{equation}
The goal of this section is to show that $\Omega_r$ remains small. Namely, using energy methods,  for some time $T$, we show that $$\sup_{t \leq T}|\Omega_r(t)|_{L^{\infty}} \leq C \alpha^{\frac{1}{2}}$$
for some constant $C$ independent of $\alpha$. 
We define the following term to shorten the notation:
$$
I_1= -\alpha R(\dth \Psi_2  +\dth \Psi_r ) (    \dR \Omega_2 +      \dR \Omega_r) \, , \, I_2=   ( 2  \Psi_2   \dth \Omega_r + 2  \Psi_r   \dth \Omega_2 + 2   \Psi_r   \dth \Omega_r  ) \, , \,  I_3=   \alpha R (   \dR \Psi_2+   \dR\Psi_r) (  \dth \Omega_2 + \dth \Omega_r) \   
$$
$$
I_4=   2 \alpha  (1-\alpha) R \sin(\theta) \cos(\theta)   (\dR \Psi_2 + \dR \Psi_r  ) \, , \, I_5=  ( 1-2 \sin^2(\theta) ) \dth \Psi_r\, , \,  I_6=   \alpha( R \cos^2(\theta) -  R \sin^2(\theta))( \dRth\Psi_2+\dRth\Psi_r)\ $$
$$
I_7=   \alpha^2(R^2 \sin(\theta) \cos(\theta)) (\dRR \Psi_2+ \dRR \Psi_r)\, , \,  I_8=  - ( \sin(\theta) \cos(\theta) ) \dthth\Psi_r\,    
$$

and now we have the error estimate proposition.
\begin{proposition}\label{ReminderEst}
Let $\Omega_r=\Omega- \Omega_2$ satisfying \eqref{ReminderEq}  with $\Omega_r|_{t=0}=0 $ then $$\sup_{0\leq t < T}|\Omega_r(t)|_{L^{\infty}} \leq c_N \alpha^{\frac{1}{2}}$$ where $T=c  \, \alpha |\log \alpha|$ and $c$  is a small constant independent of $\alpha$
 \end{proposition} 
 \begin{proof}
 We will use $\partial^{N}$ to refer to any mixed derivatives in $R$ and $\theta$ of  order $N$ (not excluding pure $R$ and $\theta$ derivatives). From the definition of $\Hn^{N}$ norm, to obtain $\Hn^{N}$ estimate we will take the following inner product with each $I_{i}$ term: 
$$\quad \big\langle \partial^{N} I_i, \partial^{N} \Omega_r \big\rangle \quad \text{and}  \quad \big\langle     R^{k} \dR^{k} \dth^{N-k}I_i ,  R^{k} \dR^{k} \dth^{N-k}  \Omega_r \big\rangle$$
for $0 \leq k \leq N$ and $1 \leq i \leq 8$.

\textbf{Estimate on $I_1$ and $I_{3}$}

Here we will estimate $I_1$ and $I_{3}$. The estimate of $I_{3}$ is very similar to $I_{1}$, and so we will just show how to obtain the estimate on $I_1$.

\textbf{Estimate on $I_1$}

We can write $I_1$ as 
\begin{equation*}
\begin{split}
I_1=   -\alpha R(\dth \Psi_2  +\dth \Psi_r ) (    \dR \Omega_2 +      \dR \Omega_r) \,     &=
-   \alpha (\dth \Psi_2   ) R(    \dR \Omega_2  ) \,             -   \alpha  (\dth \Psi_2    )  R(       \dR \Omega_r) \,     
 \\&  -   \alpha  (\dth \Psi_r ) R (    \dR \Omega_2  ) \,          -  \alpha  (\dth \Psi_r ) R (       \dR \Omega_r) \,  \\ 
 &=I_{1,1}+ I_{1,2} +I_{1,3}+ I_{1,4} \\
 \end{split}
\end{equation*}
and we will estimate each term separately.  

\begin{itemize}

\item  $I_{1,1}= -\alpha \dth \Psi_2   \,  R  \dR \Omega_2   $

Here we have 
$$ \big\langle \partial^{N}(\alpha \dth \Psi_2   \,  R  \dR \Omega_2 ), \partial^{N} \Omega_r \big\rangle =\sum_{i=0}^N c_{i,N}  \int \partial^{i} ( \alpha \dth \Psi_2 )    \partial^{N-i} ( R  \dR \Omega_2) \,   \partial^{N}\Omega_r    $$
Now from  Lemma \ref{LPsi2Est} and Lemma \ref{LOmega2Est}, we know that $$ | \Psi_2|_{\Wn^{k+1,\infty}} \leq \frac{c_{k}}{\alpha}    \quad  \text{and}   \quad    |\Omega_2|_{\Hn^{k}}  \leq   | \Omega_2(0)|_{\Hn^k} e^{ \frac{c_k}{ \alpha}t}$$

Thus, we have 
\begin{equation*}
\begin{split}
\sum_{i=0}^N\int \alpha \partial^{i} (  \dth \Psi_2 )    \partial^{N-i} ( R  \dR \Omega_2) \,   \partial^{N}\Omega_r &\leq  c_N  \sum_{i=0}^N  \alpha |\partial^{i}   \dth \Psi_2 |_{L^{\infty}} \,       |\partial^{N-i} ( R  \dR \Omega_2) |_{L^2} \,  |\partial^{N}\Omega_r |_{L^2}  \\
& \leq c_N    \alpha   | \Psi_2|_{\Wn^{N+1,\infty}}    \,      |\Omega_2|_{\Hn^{N+1}} \,       |\Omega_r|_{\Hn^N}  \leq  \alpha \,  \frac{c_{N}}{\alpha} e^{\frac{c_N}{\alpha} t}  |\Omega_r|_{\Hn^N} \\ & \leq c_N  e^{\frac{c_N}{\alpha} t}  |\Omega_r|_{\Hn^N}
 \end{split}
\end{equation*}

and similarly we have

 \begin{equation*}
\begin{split}
 \big\langle  \dR^{k} \dth^{N-k}(\alpha \dth \Psi_2   \, &  R  \dR \Omega_2 ),  R^{2k} \dR^{k} \dth^{N-k}  \Omega_r  \big\rangle =\\
 & c_{i,m,N}  \int \sum_{i+m=0}^{N}  \dR^i \dth^{m} ( \, \alpha \dth \Psi_2)  \,  \, \dR^{k-i} \dth^{N-k-m} (R \dR \Omega_{2})  \, \, R^{2k} \dR^{k} \dth^{N-k} \Omega_{r}
 \end{split}
\end{equation*}

From the definition of $\Wn^{N+1,\infty}$ norm, we have for $i+m \leq N$
$$
|R^i\dR^i \dth^{m+1}    \Psi_2 |_{L^{\infty}} \leq  | \Psi_2|_{\Wn^{N+1,\infty}}  
$$

Again, applying   Lemma \ref{LPsi2Est} and Lemma \ref{LOmega2Est},  we obtain 

\begin{equation*}
\begin{split}
\sum_{i+m=0}^{N}& \int   R^i\dR^i \dth^{m} ( \, \alpha \dth \Psi_2)  \, R^{k-i} \, \dR^{k-i} \dth^{N-k-m} (R \dR \Omega_{2})  \, \, R^{k} \dR^{k} \dth^{N-k} \Omega_{r}\\
 &\leq  c_N  \sum_{i+m=0}^{N}  \alpha |R^i\dR^i \dth^{m+1}    \Psi_2 |_{L^{\infty}} \,       | R^{k-i} \, \dR^{k-i} \dth^{N-k-m}  ( R  \dR \Omega_2) |_{L^2} \,  |R^{k} \dR^{k} \dth^{N-k} \Omega_{r} |_{L^2}  \\
& \leq c_N    \alpha  | \Psi_2|_{\Wn^{N+1,\infty}}   \,      |\Omega_2|_{\Hn^{N+1}} \,       |\Omega_r|_{\Hn^N}  \leq  \alpha \,  \frac{c_{N}}{\alpha} e^{\frac{c_N}{\alpha} t}  |\Omega_r|_{\Hn^N}  \leq c_N  e^{\frac{c_N}{\alpha} t}  |\Omega_r|_{\Hn^N}\\
 \end{split}
\end{equation*}

 Thus, we have

 \begin{equation} \label{I11H}
 \Big\langle I_{1,1}, \Omega_r  \Big\rangle_{\Hn^N} \leq  c_N  e^{\frac{c_N}{\alpha}t}  |\Omega_r|_{\Hn^N}
\end{equation}

\item  $I_{1,2}= -\alpha \dth \Psi_2   \,  R  \dR \Omega_r   $ 

Here we have 
$$ \big\langle \partial^{N}(\alpha \dth \Psi_2   \,  R  \dR \Omega_r ), \partial^{N} \Omega_r \big\rangle =\sum_{i=0}^N c_{i,N}  \int \partial^{i} ( \alpha \dth \Psi_2 )    \partial^{N-i} ( R  \dR \Omega_r) \,   \partial^{N}\Omega_r    $$
To obtain this estimate, we again apply Lemma \ref{LPsi2Est}. Namely, that $ | \Psi_2|_{\Wn^{k+1,\infty}} \leq \frac{c_{k}}{\alpha}$.  When $i=0$, we integrate by parts and obtain

$$
\int  ( \alpha \dth \Psi_2 )    \partial^{N} ( R  \dR \Omega_r) \,   \partial^{N}\Omega_r \leq c  | \Psi_2|_{\Wn^{2,\infty}}  |\Omega_r|_{\Hn^N}^2 \leq   \frac{c_N}{\alpha}  |\Omega_r|_{\Hn^N}^2
$$

For $1 \leq i\leq N$ we have,

   \begin{equation*}
\begin{split}
\sum_{i=1}^N\int \alpha \partial^{i} (  \dth \Psi_2 )    \partial^{N-i} ( R  \dR \Omega_r) \,   \partial^{N}\Omega_r &\leq  c_N  \sum_{i=1}^N  \alpha |\partial^{i}   \dth \Psi_2 |_{L^{\infty}} \,       |\partial^{N-i} ( R  \dR \Omega_r) |_{L^2} \,  |\partial^{N}\Omega_r |_{L^2}  \\
& \leq c_N    \alpha    | \Psi_2|_{\Wn^{N+1,\infty}}  \,      |\Omega_r|_{\Hn^{N}} \,       |\Omega_r|_{\Hn^N}  \leq  \alpha \,  \frac{c_{N}}{\alpha}   |\Omega_r|_{\Hn^N}^2  \leq c_N  |\Omega_r|_{\Hn^N}^2\\
 \end{split}
\end{equation*}

Similarly, Now for the $R^{k} \dR^{k} \dth^{N-k}$ terms we have

\begin{equation*}
\begin{split}
\big\langle R^{k}  \dR^{k} \dth^{N-k}(\alpha \dth \Psi_2  & \,  R  \dR \Omega_r ) ,   R^{k} \dR^{k} \dth^{N-k} \Omega_r \big\rangle = \\
&  c_{i,m,N}  \int \sum_{i+m=0}^{N}  R^{k} \dR^i \dth^{m} ( \, \alpha \dth \Psi_2)  \,  \, \dR^{k-i} \dth^{N-k-m} (R \dR \Omega_{r})  \, \, R^{k} \dR^{k} \dth^{N-k} \Omega_{r} 
 \end{split}
\end{equation*}

We again use $ | \Psi_2|_{\Wn^{k+1,\infty}} \leq \frac{c_{k}}{\alpha}$. Hence, we have

   \begin{equation*}
\begin{split}
\sum_{i+m=0}^{N}& \int   R^i\dR^i \dth^{m} ( \, \alpha \dth \Psi_2)  \, R^{k-i} \, \dR^{k-i} \dth^{N-k-m} (R \dR \Omega_{r})  \, \, R^{k} \dR^{k} \dth^{N-k} \Omega_{r}\\
 &\leq  c_N  \sum_{i+m=0}^{N}  \alpha |R^i\dR^i \dth^{m+1}    \Psi_2 |_{L^{\infty}} \,       | R^{k-i} \, \dR^{k-i} \dth^{N-k-m}  ( R  \dR \Omega_r) |_{L^2} \,  |R^{k} \dR^{k} \dth^{N-k} \Omega_{r} |_{L^2}  \\
& \leq c_N    \alpha  | \Psi_2|_{\Wn^{N+1,\infty}}   \,      |\Omega_r|_{\Hn^{N}} \,       |\Omega_r|_{\Hn^N}  \leq  \alpha \,  \frac{c_{N}}{\alpha}   |\Omega_r|_{\Hn^N}^2  \leq c_N    |\Omega_r|_{\Hn^N}^2\\
 \end{split}
\end{equation*}

 Thus, we have  
   \begin{equation} \label{I12H}
 \Big\langle I_{1,2}, \Omega_r  \Big\rangle_{\Hn^N} \leq  c_N  |\Omega_r|_{\Hn^N}^2
\end{equation}

\item  $ I_{1,3}=-   \alpha  (\dth \Psi_r ) R  \dR \Omega_2   $

To obtain the estimate on $I_{1,3}$, we will use Lemma \ref{LOmega2Est} which will give us the  estimate on $\Omega_2$. In addition, to bound the $\dth \Psi_r$ term, we will use  the elliptic from Proposition \ref{EllipticEst}  and embedding estimates from Lemma \ref{PsiEmbed}. Now we have 
$$ \big\langle \partial^{N}(\alpha \dth \Psi_r   \,  R  \dR \Omega_2 ), \partial^{N} \Omega_r \big\rangle =\sum_{i=0}^N c_{i,N}  \int \partial^{i} ( \alpha \dth \Psi_r )    \partial^{N-i} ( R  \dR \Omega_2) \,   \partial^{N}\Omega_r    $$

When $  0 \leq i \leq \frac{N}{2}   $, we will use the embedding from  Lemma \ref{PsiEmbed}. Namely that  $$  |\dR^k \dth^m \Psi|_{L^{\infty}}  \leq  c_{k,m}   |\Omega|_{\Hn^{k+m+1}} $$
Thus, we have 
 
    \begin{equation*}
\begin{split}
& \sum_{i=0}^{\frac{N}{2}}  \int \partial^{i} ( \alpha \dth \Psi_r )    \partial^{N-i} ( R  \dR \Omega_2) \,   \partial^{N}\Omega_r  \leq  \sum_{i=0}^{\frac{N}{2}}   \alpha  | \partial^{i} \dth  \Psi_r |_{L^{\infty}}|\,   |\partial^{N-i} ( R  \dR \Omega_2)|_{L^2}    |  \partial^{N} \Omega_r  |_{L^2}   \\  
& \leq       \sum_{i=0}^{\frac{N}{2}} \alpha   |\Omega_r|_{\Hn^{i+2}}   |\Omega_2|_{\Hn^{N+1}}   |\Omega_r|_{\Hn^{N} }  \leq      \alpha      |\Omega_r|_{\Hn^{\frac{N}{2}+3}}   |\Omega_2|_{\Hn^{N+1}}     |\Omega_r|_{\Hn^{N} }  \leq  c_N \alpha     e^{\frac{c_N}{\alpha}}   |\Omega_r|_{\Hn^{N} }^2   
 \end{split}
\end{equation*}
Here we used  Lemma \ref{LOmega2Est} for $ |\Omega_2|_{\Hn^{N+1}}$ term.

  When $  \frac{N}{2}  \leq i \leq N $  we will use the elliptic estimate  from Proposition \ref{EllipticEst}. Namely,
  
   $$ | \partial^k \dth \Psi|_{L^2} \, \leq c_{k} |\Omega|_{\Hn^k}$$ 
   
    Thus we have

    \begin{equation*}
\begin{split}
& \sum_{i=\frac{N}{2}}^{N}  \int \partial^{i} ( \alpha \dth \Psi_r )    \partial^{N-i} ( R  \dR \Omega_2) \,   \partial^{N}\Omega_r  \leq  \sum_{i=\frac{N}{2}}^{N}     \alpha  | \partial^{i} \dth  \Psi_r |_{L^{2}}|\,   |  R  \dR \Omega_2|_{\Wn^{N-i,\infty}}   |  \partial^{N} \Omega_r  |_{L^2}   \\  
& \leq        \sum_{i=\frac{N}{2}}^{N}    \alpha   |\Omega_r|_{\Hn^{i}}   |\Omega_2|_{\Wn^{\frac{N}{2},\infty}}    |\Omega_r|_{\Hn^{N} }  \leq      \alpha      |\Omega_r|_{\Hn^{\frac{N}{2}+3}}   |\Omega_2|_{\Hn^{N}}     |\Omega_r|_{\Hn^{N} }  \leq  c_N  \alpha     e^{\frac{c_N}{\alpha}t}   |\Omega_r|_{\Hn^{N} }^2   
 \end{split}
\end{equation*}

Similarly, to estimate the following inner product

 $$ \big\langle  \dR^{k} \dth^{N-k}(\alpha  (\dth \Psi_r ) R  \dR \Omega_2    ),  R^{2k} \dR^{k} \dth^{N-k}  \Omega_r \big\rangle \leq  c_N  \alpha     e^{\frac{c_N}{\alpha}t}   |\Omega_r|_{\Hn^{N} }^2 $$

we will use \eqref{EllipticEstW} in Proposition \ref{EllipticEst}  and embedding estimates from Lemma \ref{RPsiEmbed}. Following the same steps as we did in the previous inner product, we obtain that

\begin{equation} \label{I13H}
\Big\langle I_{1,3}, \Omega_r  \Big\rangle_{\Hn^N} \leq  c_N  \alpha     e^{\frac{c_N}{\alpha}t}   |\Omega_r|_{\Hn^{N} }^2
\end{equation}

\item Estimate on $ I_{1,4}=-   \alpha  (\dth \Psi_r ) R     \dR \Omega_r  $

To obtain the estimate on $ I_{1,4}$, we will use the elliptic estimates from  Proposition \ref{EllipticEst}. Namely,  \eqref{EllipticEstMain} and \eqref{EllipticEstW}, then we will also use the embedding estimates from Lemma  \ref{PsiEmbed} and 
 Lemma \ref{RPsiEmbed}. We will only show how to obtain the estimate on the following term:  
 
 \begin{equation*}
\begin{split}
 \big\langle \dR^k \dth^{N-k} ( \, \alpha \dth \Psi\, & R \dR \Omega_{r}) R^{2k} \dR^{k} \dth^{N-k}\Omega_{r}  \big\rangle= \\
& c_{i,m,N}  \int \sum_{i+m=0}^{N} \dR^i \dth^{m} ( \, \alpha \dth \Psi)  \,  \,  \dR^{k-i} \dth^{N-k-m} (R \dR \Omega_{r})  \, \, R^{2k} \dR^{k} \dth^{N-k} \Omega_{r}
 \end{split}
\end{equation*}

For the other inner product, the idea is the same. To start the estimate, first we consider the case when $i=m=0$, we integrate by parts and use the embedding estimates in Lemma  \ref{PsiEmbed} and 
 Lemma \ref{RPsiEmbed}, we have

 \begin{equation*}
\begin{split}
& \int \alpha \dth \Psi  \,  \Big( R^{k+1} \dR^{k+1} \dth^{N-k}  \Omega_{r}  + R^{k}  \dR^{k} \dth^{N-k}  \Omega_{r}  \Big) \, \, R^{k} \dR^{k} \dth^{N-k} \Omega_{r}   \\& \leq\alpha |R \dRth \Psi|_{L^{\infty}} \,       |R^{k} \dR^{k} \dth^{N-k} \Omega_{r}|_{L^2}^2 + \alpha | \dth \Psi|_{L^{\infty}} \,       |R^{k} \dR^{k} \dth^{N-k} \Omega_{r}|_{L^2}^2       \\ 
 & \leq c_N (    |\Omega_r|_{\Hn^3}  \,        |\Omega_r|_{\Hn^N}^2 +  |\Omega_r|_{\Hn^2}  \,        |\Omega_r|_{\Hn^N}^2)        \leq     c_N  |\Omega_r|^3_{\Hn^N}   \\
 \end{split}
\end{equation*}

 Now when $ 1 \leq i+m  \leq \frac{N}{2}$, we use Lemma \ref{RPsiEmbed} and the definition of the $\Hn^k$ norm to obtain: 
\begin{equation*}
\begin{split}
&  \sum_{i+m \geq 1}^{\frac{N}{2}}R^i \dR^i \dth^{m} ( \, \alpha \dth \Psi)  \,  \Big( R^{k+1-i} \dR^{k+1-i} \dth^{N-k-m}  \Omega_{r}  + R^{k-i}  \dR^{k-i} \dth^{N-k-m}  \Omega_{r}  \Big) \, \, R^{k} \dR^{k} \dth^{N-k} \Omega_{r}
 \\  
&  \leq \sum_{i+m \geq 1}^{\frac{N}{2}}   \alpha |R^i \, \dR^{i} \dth ^{m+1}   \Psi|_{L^{\infty}} \,   |R^{k+1-i} \dR^{k+1-i} \dth^{N-k-m}  \Omega_{r}|_{L^2}    \, \, |R^{k} \dR^{k} \dth^{N-k} \Omega_{r}|_{L^2} \,  +\\
& \quad   \, \sum_{i+m \geq 1}^{\frac{N}{2}}   \alpha |R^i \, \dR^{i} \dth ^{m+1}   \Psi|_{L^{\infty}} \,   |R^{k-i}  \dR^{k-i} \dth^{N-k-m}  \Omega_{r}|_{L^2}   \, \, |R^{k} \dR^{k} \dth^{N-k} \Omega_{r}|_{L^2} \\
& \leq   c_N    \sum_{i+m \geq 1}^{\frac{N}{2}}     |\Omega|_{\Hn^{i+m+2}} \big(  |\Omega|_{\Hn^N} + |\Omega|_{\Hn^{N-1}}\big)   |\Omega_r|_{\Hn^{N} }  \leq  c_N  |\Omega_{r}|_{\Hn^{\frac{N}{2}+2}} \big(  |\Omega|_{\Hn^N} + |\Omega|_{\Hn^{N-1}}\big)  | \Omega_{r}|_{\Hn^N }   \\ 
& \leq  c_N | \Omega_{r}|_{\Hn^N }^3
 \end{split}
\end{equation*}

Now for the case when $ \frac{N}{2} \leq i+m  \leq N$, we will use Lemma \ref{OmegaEmbed} and the elliptic estimates from Proposition \ref{EllipticEst} to obtain

\begin{equation*}
\begin{split}
&  \sum_{i+m \geq \frac{N}{2}}^{N}R^i \dR^i \dth^{m} ( \, \alpha \dth \Psi)  \,  \Big( R^{k+1-i} \dR^{k+1-i} \dth^{N-k-m}  \Omega_{r}  + R^{k-i}  \dR^{k-i} \dth^{N-k-m}  \Omega_{r}  \Big) \, \, R^{k} \dR^{k} \dth^{N-k} \Omega_{r}
 \\  
&  \leq \sum_{i+m \geq \frac{N}{2}}^{N}   \alpha |R^i \, \dR^{i} \dth ^{m+1}   \Psi|_{L^2} \,  \Big(  |R^{k+1-i} \dR^{k+1-i} \dth^{N-k-m}  \Omega_{r}|_{L^{\infty}}  \Big) \, \, |R^{k} \dR^{k} \dth^{N-k} \Omega_{r}|_{L^2} \, + \\
&  \quad  \sum_{i+m \geq \frac{N}{2}}^{N}   \alpha |R^i \, \dR^{i} \dth ^{m+1}   \Psi|_{L^2} \,  \Big( |R^{k-i}  \dR^{k-i} \dth^{N-k-m}  \Omega_{r}|_{L^{\infty}}  \Big) \, \, |R^{k} \dR^{k} \dth^{N-k} \Omega_{r}|_{L^2} \\
& \leq \sum_{i+m \geq \frac{N}{2}}^{N}  |\Omega_{r}|_{\Hn^{i+m-1}} \,  \big(  |\Omega|_{\Hn^{N-(i+m)+3}} +  |\Omega|_{\Hn^{N-(i+m)+2}}  \big) \,  | \Omega_{r}|_{\Hn^N } \leq  c_N   |\Omega_{r}|_{\Hn^{N-1}} |\Omega|_{\Hn^{\frac{N}{2}+3}}       | \Omega_{r}|_{\Hn^N } \\ 
& \leq c_N | \Omega_{r}|_{\Hn^N }^3
 \end{split}
\end{equation*}

and thus, we have the following: 
\begin{equation} \label{I14H}
 \Big\langle I_{1,4}, \Omega_r  \Big\rangle_{\Hn^N} \leq  c_N  |\Omega_r|_{\Hn^N}^3
\end{equation}

\end{itemize}

Thus, we have the following estimate on $I_1$ term 

\begin{equation} \label{I1H}
 \Big\langle I_{1}, \Omega_r  \Big\rangle_{\Hn^N} \leq c_N  e^{\frac{c_N}{\alpha}t}  |\Omega_r|_{\Hn^N}+ c_N       e^{\frac{c_N}{\alpha}t}   |\Omega_r|_{\Hn^{N} }^2
+c_N  |\Omega_r|_{\Hn^N}^3
\end{equation}

\textbf{Estimate on $I_3$}

The estimate on $I_3$ follows similarly to $I_1$, so we skip the details for this case.  One can obtain the following: 
\begin{equation} \label{I3H}
 \Big\langle I_{3}, \Omega_r  \Big\rangle_{\Hn^N} \leq c_N  e^{\frac{c_N}{\alpha}t}  |\Omega_r|_{\Hn^N}+ c_N       e^{\frac{c_N}{\alpha}t}   |\Omega_r|_{\Hn^{N} }^2
+c_N  |\Omega_r|_{\Hn^N}^3
\end{equation}

\textbf{Estimate on $I_2$}

Here we have 
$$ I_2=   ( 2  \Psi_2   \dth \Omega_r + 2  \Psi_r   \dth \Omega_2 + 2   \Psi_r   \dth \Omega_r  ) \  =  I_{2,1}+ I_{2,2}+I_{2,3} $$

\begin{itemize}

\item $I_{2,1}= 2 \Psi_2   \dth \Omega_r$

To estimate $I_{2,1}$, we follow the same steps as in the $I_1$ term. Using Lemma \ref{LPsi2Est}. Namely, that $ | \Psi_2|_{\Wn^{N,\infty}} \leq \frac{c_{N}}{\alpha}$, we have:

\begin{equation} \label{I21H}
\Big\langle I_{2,1}, \Omega_r  \Big\rangle_{\Hn^N} \leq   \frac{c_N}{\alpha}  |\Omega_r|_{\Hn^N}^2
 \end{equation}

\item  $I_{2,2}= 2  \Psi_r   \dth \Omega_2 $

Similarly, To estimate $I_{2,2}$, we also follow the same steps as we did in $I_1$. Using Lemma \ref{LOmega2Est}, that $  |\Omega_2|_{\Hn^{k}}  \leq   | \Omega_2(0)|_{\Hn^k} e^{ \frac{c_k}{ \alpha}t}$, we obtain:      
\begin{equation} \label{I22H}
\Big\langle I_{2,2}, \Omega_r  \Big\rangle_{\Hn^N} \leq   c_N      e^{\frac{c_N}{\alpha}t}   |\Omega_r|_{\Hn^{N} }^2  
\end{equation}

\item$I_{2,3}= 2   \Psi_r   \dth \Omega_r$

This terms $I_{2,3}$ can be estimated in a similar way as in the $I_{1,4}$ term, by using embedding and elliptic estimate, we have

\begin{equation} \label{I23H}
\Big\langle I_{2,3}, \Omega_r  \Big\rangle_{\Hn^N} \leq    c_N  |\Omega_r|_{\Hn^N}^3  
\end{equation}

\end{itemize}

Hence, we obtain 

\begin{equation} \label{I2H}
\Big\langle I_{2}, \Omega_r  \Big\rangle_{\Hn^N} \leq  \frac{c_N}{\alpha}  |\Omega_r|_{\Hn^N}^2 + c_N      e^{\frac{c_N}{\alpha}t}   |\Omega_r|_{\Hn^{N} }^2  +c_N  |\Omega_r|_{\Hn^N}^3 \leq  c_N      e^{\frac{c_N}{\alpha}t}   |\Omega_r|_{\Hn^{N} }^2  +c_N  |\Omega_r|_{\Hn^N}^3
\end{equation}

\textbf{Estimate on $I_4$, $I_5$, $I_6$, $I_7$, and $I_8$}

We can write $I_4$ as: 
\begin{equation*}
\begin{split}
 I_4&=  2 \alpha R \sin(\theta) \cos(\theta) +\alpha^2 R \sin(\theta) \cos(\theta) )  (\dR \Psi_2 + \dR \Psi_r  )\\
    & =  \alpha(2+\alpha) \,\sin(\theta) \cos(\theta) \,  R \dR \Psi_2 +  \alpha(2+\alpha) \,\sin(\theta) \cos(\theta) \,  R \dR \Psi_r\\
    &=I_{4,1}+ I_{4,2}
    \end{split}
\end{equation*}

Recall that $$I_{5}=(1-2\sin^2(\theta)) \dth \Psi_r$$. 

We can also rewrite and $I_{6}$ and $I_{7}$ as follows:

\begin{equation*}
\begin{split}
 I_{6}&= \alpha(  \cos^2(\theta) - \sin^2(\theta)) R( \dRth\Psi_2+\dRth\Psi_r)    \\
    & = \alpha(  \cos^2(\theta) -  \sin^2(\theta)) R \dRth\Psi_2 + \alpha(  \cos^2(\theta) -  \sin^2(\theta)) R \dRth\Psi_r\\
    &=I_{6,1}+I_{6,2} 
    \end{split}
\end{equation*}
and 
\begin{equation*}
\begin{split}
 I_7&= \alpha^2 (\sin(\theta) \cos(\theta))R^2 (\dRR \Psi_2+ \dRR \Psi_r)   \\
    & = \alpha^2 (\sin(\theta) \cos(\theta))R^2 \dRR \Psi_2 +\alpha^2 (\sin(\theta) \cos(\theta))R^2 \dRR \Psi_r \\
    &=I_{7,1}+I_{7,2} 
    \end{split}
\end{equation*}
Recall that $$I_{8}=-\sin(\theta) \cos(\theta) \dthth \Psi_r$$.

Now for $i=4, 6, \text{and} \, 7$,  using Lemma \ref{LPsi2Est}. Namely, that $|\Psi|_{\Hn^{k+1}} \leq  \frac{c_{k}}{ \alpha}$, we have the following estimate: 
\begin{equation} \label{Ii1H}
\big\langle I_{i,1} \,  ,  \, \Omega_r \big\rangle_{\Hn^{N}} \leq   c_{N} |\Omega_r|_{\Hn^{N}}  \quad \text{for}\quad i=4,6,7
\end{equation}
Using  the elliptic estimates in Proposition  \ref{EllipticEst}, we   obtain that: 

\begin{equation} \label{Ii2H}
\big\langle I_{i,2} \,  ,  \, \Omega_r \big\rangle_{\Hn^{N}} \leq c_N |\Omega_r|_{\Hn^{N}}^2  \quad \text{for}\quad i=4,6,7
\end{equation}
and 
\begin{equation} \label{I58H}
\big\langle I_{i} \,  ,  \, \Omega_r \big\rangle_{\Hn^{N}} \leq c_N |\Omega_r|_{\Hn^{N}}^2  \quad \text{for}\quad i=5,8
\end{equation}
Hence, from \eqref{Ii1H}, \eqref{Ii2H}, \eqref{I58H}, we have  that 

\begin{equation} \label{I45678H}
\big\langle I_{i} \,  ,  \, \Omega_r \big\rangle_{\Hn^{N}} \leq c_N |\Omega_r|_{\Hn^{N}} +c_N |\Omega_r|_{\Hn^{N}}^2  \quad \text{for}\quad i=4,5, \dots 8
\end{equation}

\textbf{Totally reminder  estimate:}

Here we obtain the totally error estimate. From our previous work we have,   
\begin{equation*}
\begin{split}
\frac{d}{dt} |\Omega_r|^2_{\Hn^{N}} =  \big\langle \partial_t  \Omega_r  ,  \, \Omega_r \big\rangle_{\Hn^{N}}  &\leq   \sum_{i=1}^8 |   \, \big\langle  I_{i}  ,  \, \Omega_r \big\rangle_{\Hn^{N}}  | 
 \\
\end{split}
\end{equation*}

and thus from \eqref{I1H}, \eqref{I3H}, \eqref{I2H}, and   \eqref{I45678H}  , we have 
\begin{equation*}
\begin{split}
\frac{d}{dt}  |\Omega_r|^2_{\Hn^{N}} &\leq    
      c_N  e^{\frac{c_N}{\alpha}t}  |\Omega_r|_{\Hn^N}+\Big( \frac{c_N}{\alpha}  +c_N      e^{\frac{c_N}{\alpha}t} \Big) \,  |\Omega_r|_{\Hn^{N} }^2   
+c_N  |\Omega_r|_{\Hn^N}^3    \\    
\end{split}
\end{equation*}
and hence 

\begin{equation*}
\begin{split}
\frac{d}{dt}  |\Omega_r|_{\Hn^{N}} &\leq    
      c_N  e^{\frac{c_N}{\alpha}t}+ \Big( \frac{c_N}{\alpha}  +c_N      e^{\frac{c_N}{\alpha}t} \Big) \,  |\Omega_r|_{\Hn^{N} }  
+c_N  |\Omega_r|_{\Hn^N}^2    \\    
\end{split}
\end{equation*}
Now since we have $\Omega_r|_{t=0}=0$, then by bootstrap,  it follows that

\begin{equation*}
\begin{split}
  |\Omega_r|_{\Hn^{N}} &\leq    
      \Big( \int_0^t c_N  e^{\frac{c_N}{\alpha}\tau} \, d\tau \Big) \exp(\int_0^t  \frac{c_N}{\alpha}  +c_N      e^{\frac{c_N}{\alpha}\tau} \, d\tau) \leq  \alpha c_N (e^{\frac{c_N}{\alpha} t} -1) \exp(  \frac{c_N}{\alpha} t  + \alpha c_N      e^{\frac{c_N}{\alpha}t} \, ) \\    
\end{split}
\end{equation*}

Thus, if we choose $t <  T= c \, \alpha |\log \alpha|$ for $c$ small, say $c=\frac{1}{4 \, c_N}$,  then we have 
\begin{equation*}
\begin{split}
  |\Omega_r|_{\Hn^{N}} &\leq    
     c_N  \alpha^{\frac{1}{2}} \, \\    
\end{split}
\end{equation*}

and this completes the proof of  Proposition \ref{ReminderEst}

 \end{proof} 
 
 \section{ Main result  }\label{Main}
 
We now recall and prove the main theorem of this work. 

\begin{theorem}
For any $\alpha,\delta>0$, there exists an initial data $\omega_0^{\alpha,\delta} \in C_c^{\infty}(\mathbb{R}^2)$  and   $T(\alpha)$ such that the corresponding unique global solution, $\omega^{\alpha,\delta}$, to \eqref{EulerR} is such that 
  at $t=0$ we have $$|\omega_0^{\alpha,\delta}|_{L^{\infty}}=\delta,$$ but for any  $0<t\leq T(\alpha)$ we have $$    \quad|\omega^{\alpha,\delta}(t)|_{L^{\infty}} \geq |\omega_0|_{L^\infty}+c \log (1+\frac{c}{\alpha}t),
  $$
where $T(\alpha)= c \alpha |\log(\alpha)|$ and $c>0$ is a constant independent of $\alpha.$
\end{theorem}
\begin{proof}

Consider the initial data of the form 

$$
\omega_0=\Omega|_{t=0}=f_{0}(R) \sin(2\theta)
$$
where $f(R)$ is non-negative compactly support  smooth function which is zero on $[0,1)$ and positive outside. We know that we can write  $\Omega=\Omega_2 +\Omega_r$, and from the form of initial data, we have $\Omega_r|_{t=0}=0$ and thus from Proposition \ref{ReminderEst}, we have 

$$
  |\Omega_r (t)|_{L^{\infty}} \leq    
    c_N   \alpha^{\frac{1}{2}} \, 
$$  

for $0\leq t \leq c \,  \alpha |\log \alpha|$, where recall that $c$ is a small constant independent of $\alpha$. 
Recall also that we can write $\Omega_2$ as:
$$
 \Omega_2= f + \frac{1}{2 \alpha} \int_0^t  L_s(f_\tau) d \tau
$$
and thus from Proposition \ref{LeadingEst}, we obtain that 

$$
 \Omega_2= f + \frac{1}{2 \alpha} \int_0^t  L_s(f_\tau) d \tau \geq f+ c_0  \log(1+ \frac{c_0}{\alpha}  \, t ))
$$
for some $c_0$ independent of  of $ \alpha$ and thus we have our desired result.

\end{proof}

\section*{Acknowledgements}
Both authors were partially supported by the NSF grants DMS 2043024 and DMS 2124748. T.M.E. was partially supported by an Alfred P. Sloan Fellowship.

\vskip 0.3 cm
Department of Mathematics, Duke University, Durham, NC 27708, USA

\textit{Email address:} tarek.elgindi@duke.edu

Department of Mathematics, Duke University, Durham, NC 27708, USA

\textit{Email address:} karimridamoh.shikh.khalil@duke.edu

\end{document}